\documentclass[10pt]{amsart}
\usepackage[margin=1.5in]{geometry}
\usepackage{color}
\newtheorem{theorem}{Theorem}[section]
\newtheorem{lemma}[theorem]{Lemma}
\newtheorem{proposition}{Proposition}[section]

\theoremstyle{definition}
\newtheorem{definition}[theorem]{Definition}

\theoremstyle{remark}
\newtheorem{remark}[theorem]{Remark}

\numberwithin{equation}{section}

\newcommand{\rr}{\mathbb R}
\newcommand{\rrr}{\mathbb R^3}

\newcommand{\bq}{\begin{equation}}
\newcommand{\eq}{\end{equation}}

\newcommand{\mm}{\mathcal{M}_{\nu}}

\usepackage{amsmath,amsfonts,amssymb,amsthm}
\usepackage{mathtools}
\usepackage{commath}

\begin{document}

\title[Stationary Flows of the ES-BGK model in a slab ]{Stationary Flows of the ES-BGK model with the correct Prandtl number} 

\author{Stephane Brull}
\address{Institut de Mathématiques de Bordeaux UMR 5251
Université Bordeaux
351, cours de la Libération
33405 TALENCE cedex FRANCE  }
\email{Stephane.Brull@math.u-bordeaux1.fr }
\author{Seok-Bae Yun}
\address{Department of Mathematics, Sungkyunkwan University, Suwon 440-746, Republic of Korea}
\email{sbyun01@skku.edu}




\begin{abstract}
Ellipsoidal BGK model (ES-BGK) is a generalized version of the BGK model  where the local Maxwellian in the relaxation operator of the BGK model is extended to an ellipsoidal 
Gaussian with a Prandtl parameter $\nu$, so that the correct transport coefficients can be computed in the Navier-Stokes limit. 
In this work, we consider the existence and uniqueness of stationary solutions for the ES-BGK model in a slab imposed with the mixed boundary conditions.
One of the key difficulties arise in the uniform control of the temperature tensor from below.  
In the non-critical case $(-1/2<\nu<1)$, we utilize the property that the temperature tensor is
equivalent to the temperature in this range. 
In the critical case, $(\nu=-1/2)$, where such equivalence relation breaks down, we observe that the size of bulk velocity in $x$ direction can be controlled by the discrepancy of boundary flux, 
which enables one to bound the temperature tensor from below. 
\end{abstract}
\maketitle
\section{introduction}
\subsection{Ellipsoidal BGK model:} The Boltzmann equation is a fundamental model that connects the particle regime  and the fluid regime of rarefied gases. It has, however, not been as practical an equation as it is a fundamental equation in that the application of the Boltzmann equation to various flow problems has
been severely restricted by the intricate structure of the collision operator that requires a serious amount of resource for numerical computations.
The observation made by Bhatnaghar, Gross and Krook \cite{BGK} in their attempts to overcome this difficulty is that
the local equilibration occurs rather quickly, so that the complicated process of collision can be successfully described by the relaxation process after a short time scale. 
The equation that was introduced based on this observation is now called the BGK model, and  have enjoyed a great popularity as a numerically amenable equation that provides qualitatively satisfactory results.
There are, however, several shortcommings of the model. Most notable one is that the Prandtl number - the ratio between the thermal diffusivity and the viscosity,  computed from the BGK model does not
match the correct value computed from the Boltzmann equation, which means that the diffusivity and 
the viscosity in the Navier-Stokes limit cannot be correctly derived. 
In this regards, Holway proposed so-called the ellipsoidal BGK model (ES-BGK model), which generalizes the local Maxwellian of the BGK model to an ellipsoidal Gaussian endowed with an additional degree of freedom in adjusting the transport coefficients.
ES-BGK model, however, was somewhat forgotten in the literature since it was not clear at the time whether the H-theorem holds for this model. This was resolved  by Andries et al in \cite{ALPP} (and later in \cite{B3,PY3}), which greatly popularized this model in the study of various problems in the rarefied gas dynamics. The existence result of the ES-model in the critical case ($\nu=-1/2$), however, was never
made so far except for the case where the solution lies close to equilibrium \cite{Yun22, Yun3}, which is the main motivation of the current work.

More precisely, we are interested in the boundary value problem of the stationary ellipsoidal BGK model:
\begin{equation}\label{ESBGK}
  v_{1} \frac{\partial f}{\partial x} =\frac{1}{\kappa(1-\nu)}\big( \mathcal{M}_{\nu}(f)-f \big),
\end{equation}
on a finite interval $[0,1]$ where the boundary condition is given by the linear combination of 
the inflow boundary condition, the diffusive boundary condition, and the specular reflection $(\delta_1+\delta_2+\delta_3=1)$:
\begin{align}\label{boundary conditions}
\begin{split}
f(0,v)&=\delta_1f_L(v)+\delta_2\left(\int_{|v_1|<0}f(0,v)|v_1|dv\right)M_w(0,v)+\delta_3f(0,Rv),\quad (v_1>0)\cr
f(1,v)&=\delta_1f_R(v)+\delta_2\left(\int_{|v_1|>0}f(1,v)|v_1|dv\right)M_w(1,v)+\delta_3 f(1,Rv).\quad (v_1>0)
\end{split}
\end{align}
Here $M_w$ denotes the wall Maxwellians which, for a given wall temperature $T_w:\{0,1\}\rightarrow\mathbb{R}_+$, is defined by
\[
M_w(i,v)=\frac{1}{\sqrt{2\pi T_w(i)}}e^{-\frac{|v|^2}{2T_w(i)}}.\qquad (i=1,2)
\]
When there's no risk of confusion we denote both $M_w(0,v)$ and $M_w(1,v)$ by $M_w$. $Rv$ denotes the reflection of $v$: $R(v_1,v_2,v_3)=(-v_1,v_2,v_3)$. 
  We note that $\delta_1$ term and $\delta_2$ term corresponds to the condensation and the evaporation
  at the boundary \cite{SAD}.

The velocity distribution function $f(x,v)$ represents the number density of the gas molecules at the position $x\in[0,1]$
with the microscopic velocity $v=(v_1,v_2,v_3)\in\mathbb{R}^3$.  $\kappa$ is the Knudsen number. The ellipsoidal Gaussian $\mathcal{M}_{\nu}(f)$ with the Prandtl parameter $\nu\in [-1/2,1)$ reads
\begin{eqnarray*}
\mathcal{M}_{\nu}(f)=\frac{\rho}{\sqrt{\det(2\pi \mathcal{T}_{\nu})}}\exp\left(-\frac{1}{2}(v-U)^{\top}\mathcal{T}^{-1}_{\nu}(v-U)\right).
\end{eqnarray*}
The local density $\rho$, momentum $U$, temperature $T$ and the stress tensor $\Theta$ are given by the following relations:
\begin{eqnarray}\label{macroscopic field}
\begin{split}
\rho(x)&=\int_{\mathbb{R}^3}f(x,v)dv,\cr
\rho(x)U(x)&=\int_{\mathbb{R}^3}f(x,v)vdv,\cr
3\rho(x) T(x)&=\int_{\mathbb{R}^3}f(x,v)|v-U|^2dv,\cr
\rho(x)\Theta(x)&=\int_{\mathbb{R}^3}f(x,v)(v-U)\otimes(v-U)dv,
\end{split}
\end{eqnarray}
and the temperature tensor $\mathcal{T}_{\nu}$ is defined as a linear combination of the temperature and the stress tensor:
\begin{align*}
\mathcal{T}_{\nu}
&=(1-\nu) T \mathbb{I}+\nu\Theta\cr
&=\left(
\begin{array}{ccc}
(1-\nu) T+\nu\Theta_{11}&\nu\Theta_{12}&\nu\Theta_{13}\cr
\nu\Theta_{21}&(1-\nu)T+\nu\Theta_{22}&\nu\Theta_{23}\cr
\nu\Theta_{31}&\nu\Theta_{32}&(1-\nu) T+\nu\Theta_{33}
\end{array}
\right)
\end{align*}
where $\mathbb{I}_3$ denotes the $3\times 3$  identity matrix. 

Note that this is not a convex combination since $\nu$ can take negative values.
In the case $\nu=0$, the ES-BGK model reduces to the original BGK model. The end-point $\nu=-1/2$ corresponds to ES-BGK model with the correct Prandtl number: The Prandtl number
computed using the ES-BGK model with $\nu=-1/2$ matches with the correct value computed from the Boltzmann equation.
For simplicity, we set $\tau=\kappa(1-\nu)$ throughout the paper and write (\ref{ESBGK}) as
\begin{equation*}
v_{1} \frac{\partial f}{\partial x} =\frac{1}{\tau}\big( \mathcal{M}_{\nu}(f)-f \big).
\end{equation*}
%
%
%
%
%
\subsection{Notations:}

%
%
%
%
We first set up notational conventions and define norms:
\begin{itemize}
\item $C$ denote generic constants. The value can change each line of computations, but it 
is explicitly computable in principle. 
\item $A\preceq B$ means that $A\leq C B$ for some constant $C$. 
\item $\mathbb{I}_3$ denotes the $3\times 3$  identity matrix.
\item We define $f_{LR}$ and $M_w$ by
\[f_{LR}(v)=f_{L}(v)1_{v_1>0}+f_{R}(v)1_{v_1<0}\]
and
\[M_w(v)=M_w(0,v)1_{v_1>0}+M_w(1,v)1_{v_1<0}.\]
\item We define $\sup_x\|\cdot\|_{L^1_2}$ by
\begin{align*}
\sup_x\|f\|_{L^1_2}&=\sup_x\|f\|_{L^1_{2,+}}+\sup_x\|f\|_{L^1_{2,-}},
\end{align*}
where
\begin{align*}
	\sup_x\|f\|_{L^1_{2,+}}&=\sup_x\Big\{\int_{v_1>0}|f(x,v)|(1+|v|^2)dv\Big\},\cr
	\sup_x\|f\|_{L^1_{2,-}}&=\sup_x\Big\{\int_{v_1<0}|f(x,v)|(1+|v|^2)dv\Big\}.
\end{align*}
\item We define the trace norm $ \|\cdot\|_{L^1_{\gamma,|v_1|}}$ by
\begin{align*}
\|f\|_{L^1_{\gamma,|v_1|}}=\|f\|_{L^1_{\gamma,|v_1|,+}}+\|f\|_{L^1_{\gamma,|v_1|,-}},
\end{align*}
where the outward trace norm $\|f\|_{L^1_{\gamma,|v_1|,+}}$ and the inward trace norm $\|f\|_{L^1_{\gamma,|v_1|,-}}$ are given by
\begin{align*}
	\|f\|_{L^1_{\gamma,|v_1|,+}}
	&=\int_{v_1<0}|f(0,v)||v_1|dv+\int_{v_1>0}|f(1,v)||v_1|dv,\cr
	\|f\|_{L^1_{\gamma,|v_1|,-}}
	&=\int_{v_1>0}|f(0,v)||v_1|dv+\int_{v_1<0}|f(1,v)||v_1|dv.
\end{align*}
\item Throughout the paper, we normalize the wall Maxwellian as follows:
\[
\|M_w\|_{L^1_{\gamma,|v_1|,\pm}}=1.
\]
\item Similarly, we define another trace norm $ \|\cdot\|_{L^1_{\gamma,\langle v \rangle}}$ by
\begin{align*}
\|f\|_{L^1_{\gamma,\langle v\rangle}}=\|f\|_{L^1_{\gamma,\langle v\rangle,+}}+\|f\|_{L^1_{\gamma,\langle v\rangle,-}}.
\end{align*}
where the outward trace norm $\|\cdot\|_{L^1_{\gamma,\langle v\rangle},+}$ and the inward trace norm $\|\cdot\|_{L^1_{\gamma,\langle v\rangle},-}$ are given  by
\begin{align*}
	\|f\|_{L^1_{\gamma,\langle v\rangle,+}}&=\int_{v_1<0}|f(0,v)|(1+|v|^2)dv+\int_{v_1>0}|f(1,v)|(1+|v|^2)dv,\cr
	\|f\|_{L^1_{\gamma,\langle v\rangle,-}}&=\int_{v_1>0}|f(0,v)|(1+|v|^2)dv+\int_{v_1<0}|f(1,v)|(1+|v|^2)dv.
\end{align*}
\item Throughout the paper, $C_{LR,1}$, $C_{LR,2}$  denote
\begin{align}\label{CLR}
	\begin{split}
C_{LR,1}&=\|f_{LM}\|_{L^1_{\gamma,\langle v_1\rangle} }\|M_{w}\|_{L^1_{\gamma,\langle v\rangle}},\cr
C_{LR,2}&=\|f_{LR}\|_{L^1_{\gamma,|v_1|}}+\|M_w\|_{L^1_{\gamma,\langle v\rangle}},
\end{split}
\end{align}
and $a_{\ell,1}$ and $a_{\ell,2}$ denote
\begin{align}\label{quantities}
	\begin{split}
		a_{\ell,1}=\int_{\mathbb{R}^3}e^{- \frac{ 1}{|v_1|}}f_{LR}dv,\qquad
		a_{\ell,2}=\frac{1}{2}\int_{\mathbb{R}^3}e^{- \frac{ 1}{|v_1|}}M_wdv.
	\end{split}
\end{align}
\end{itemize}
\noindent{\bf(P) Properties of boundary data:} To avoid repetition in the statement of the theorem,  we summarize here the assumptions to be imposed on $f_{LR}$ later:\newline

\noindent$(P_1)$ The inflow boundary data $f_{LR}\geq 0$, not identically $0$, has a finite trace norm:
\begin{eqnarray*}
\|f_{LR}\|_{L^{1}_{\gamma,\langle v\rangle}}<\infty
\end{eqnarray*}

\noindent$(P_2)$ The inflow data does not induce vertical flows: :
\begin{equation*}
\int_{\mathbb{R}^2}f_Lv_idv=\int_{\mathbb{R}^2}f_Rv_3dv=0 \quad(i=2,3)
\end{equation*}

\subsection{Main result 1: inflow dominant case} We now state our main results for the inflow dominant case. 
That is, when $\delta_1$ is not small.
We first define the mild solution of (\ref{ESBGK}) for the inflow dominant case (Theorem \ref{Main1}) as follows:
\begin{definition} $f\in L^{\infty}\left([0,1]; L^1_2(\mathbb{R}^3)\right)\cap L^1_{\gamma,\langle v\rangle}(\mathbb{R}^3)$ is said to be a mild solution for (\ref{ESBGK})
	if it satisfies
	\begin{align}\label{mild f+}
		\begin{split}
		f(x,v)&=e^{-\frac{x}{\tau|v_1|}}f(0,v)
		+\frac{1}{\tau|v_1|}\int_{0}^{x}
                e^{-\frac{x-y}{\tau|v_1|}}\mathcal{M}_{\nu}(f)dy
		\quad\text{if $v_{1}>0$}
		\end{split}
	\end{align}
	and
	\begin{align}\label{mild f-}
		\begin{split}
		f(x,v)&=e^{-\frac{x}{\tau|v_1|}\int^1_xdy}f(1,v)
		+\frac{1}{\tau|v_1|}\int_{x}^1
                e^{-\frac{x-y}{\tau|v_1|}}
		\mathcal{M}_{\nu}(f)dy
		\quad\text{if $v_{1}<0$},
		\end{split}
	\end{align}
	where $f(0,v)$ and $f(1,v)$ are defined in the trace sense as
	\begin{align*}
		\begin{split}
			f(0,v)&=\delta_1f_L(v)+\delta_2\left(\int_{|v_1|<0}f(0,v)|v_1|dv\right)M_w(0)+\delta_3
                        f(0,Rv),\quad (v_1>0) ,  \cr 
			f(1,v)&=\delta_1f_R(v)+\delta_2\left(\int_{|v_1|>0}f(1,v)|v_1|dv\right)M_w(1)+\delta_3
                        f(1,Rv),\quad (v_1<0) .
		\end{split}
	\end{align*}
\end{definition}
We are now ready to state the main result of this paper:

\begin{theorem}\label{Main1}{\bf [Inflow dominant case]}\newline
$(1)$ $($Non-critical $\nu$$)$ Let $-1/2<\nu<1$. Suppose $f_{LR}$ satisfies $(P_1)$ and $(P_2)$.
Then there exist  constants $K_1>0$, $\epsilon>0$ such that, if  $\tau>K_1$ and $\delta_2+\delta_3<\epsilon$,
then there exists a unique mild solution $f\geq0$ to the boundary value problem  (\ref{ESBGK}), (\ref{boundary conditions}) satisfying
\begin{align*}
\int_{\mathbb{R}^3}f(x,v)dv\geq a_{\ell,1},\quad \int_{\mathbb{R}^3}f(x,v)(1+|v|^2)dv\leq2C_{LR,1},
\end{align*}
and

\begin{align*}
	C^1_{\nu}\delta_1^2\frac{\gamma_{\ell,1}}{3C_{LR,1}^2}\leq\kappa^{\top}\left\{\mathcal{T}_{\nu}\right\}\kappa\leq \frac{2 }{{3a_{\ell,1}}\delta_1}C^2_{\nu}C_{LR,1},
\end{align*}
where $\gamma_{\ell,1}$ is defined by
\begin{align}\label{gamma}
\gamma_{\ell,1}=\left(\int_{v_1>0}e^{-\frac{1}{|v_1|}}f_L(v)|v_1|dv\right)
\left(\int_{v_1<0}e^{-\frac{1}{|v_1|}}f_R(v)|v_1|dv\right)>0.
\end{align}
\noindent$(2)$ $($Critical $\nu$$)$ Let $\nu=-1/2$:
Suppose $f_{LR}$ satisfies $(P_1)$ and $(P_2)$. 
Then there exist  constants $K_1>0$, $\varepsilon_1, \varepsilon_2>0$  such that, if  $\tau>K_1$, $\delta_2+\delta_3<\epsilon_1$ and 
\[
\left|\int_{v_1>0}f_L|v_1|dv-\int_{v_1<0}f_R|v_1|dv\right|\leq\varepsilon_2,
\]
then there exists a unique mild solution $f\geq0$ to the boundary value problem  (\ref{ESBGK}), (\ref{boundary conditions}) satisfying
\begin{align*}
\int_{\mathbb{R}^3}f(x,v)dv\geq \delta_1a_{\ell,1},\quad \int_{\mathbb{R}^3}f(x,v)|v|^2dv\leq 
2C_{LR,1},
\end{align*}
and
\begin{align*}
	\delta_1\frac{a_{-1/2,1}}{2C_{LR,1}}\leq\kappa^{\top}\left\{\mathcal{T}_{-1/2}\right\}\kappa\leq \frac{3}{2a_{\ell,1}}C_{LR,1},
\end{align*}
where $a_{-1/2,1}$ denote
\begin{align}\label{a1}
a_{-1/2,1}=\inf_{|\kappa|=1}
\int_{\mathbb{R}^3}e^{-\frac{1}{|v_1|}
}f_{LR}\left\{|v|^2-(v\cdot\kappa)^2\right\}dv>0.
\end{align}
\end{theorem}
\subsection{Main results 2: diffusive dominant case} In the diffusive dominant case, that is, when the the diffusive boundary condition dominates so that we cannot enforce smallness on $\delta_2$, we impose the following flux control condition:
\begin{align}\label{flux c}
\int_{v_1<0}f(0,v)|v_1|dv+\int_{v_1>0}f(1,v)|v_1|dv=1.
\end{align}
This is because, without such additional assumptions, we generally  don't have uniquess for the boundary problem with diffusive boundary conditions (See the paragraphs following the Theorem \ref{Main2}.)
Using (\ref{ESBGK}), (\ref{boundary conditions}) and (\ref{flux c}), we reformulate the boundary
condition into (\ref{new boundary}), and we consider 
the following mild solution for the diffusive dominant case (Theorem \ref{Main2}).
(See Section 7 for the detail of the reformulation.)
\begin{definition} $f\in L^{\infty}\left([0,1]; L^1_2(\mathbb{R}^3)\right)\cap L^1_{\gamma,\langle v\rangle}(\mathbb{R}^3)$ is said to be a mild solution for (\ref{ESBGK})
	if it satisfies
	\begin{align}\label{mild f+dff}
		\begin{split}
		f(x,v)&=e^{-\frac{x}{\tau|v_1|}}f(0,v)
		+\frac{1}{\tau|v_1|}\int_{0}^{x}
                e^{-\frac{x-y}{\tau|v_1|}}\mathcal{M}_{\nu}(f)dy
		\quad\text{if $v_{1}>0$}
		\end{split}
	\end{align}
	and
	\begin{align}\label{mild f-dff}
		\begin{split}
		f(x,v)&=e^{-\frac{x}{\tau|v_1|}}f(1,v)
		+\frac{1}{\tau|v_1|}\int_{x}^1
                e^{-\frac{x-y}{\tau|v_1|}}
	\mathcal{M}_{\nu}(f)dy
		\quad\text{if $v_{1}<0$},
		\end{split}
	\end{align}
	where
	\begin{align}\label{new boundary}
		\begin{split}
			f(0,v)&=\delta_1f_L(v)+\delta_2\mathcal{S}_L(f)M_w(0)+\delta_3 f(0,Rv),\quad (v_1>0)\cr
			f(1,v)&=\delta_1f_R(v)+\delta_2\mathcal{S}_R(f)M_w(1)+\delta_3 f(1,Rv),\quad (v_1<0)
		\end{split}
	\end{align}
	and $\mathcal{S}_L(f)$, $\mathcal{S}_R(f)$ denote
	\begin{align*}
		\mathcal{S}_L(f)&=\frac{1-\delta_1}{2-\delta_1}+\frac{\delta_1}{2-\delta_1}\int_{v_1<0}f_R|v_1|dv
		-\frac{1}{\tau(2-\delta_1)}\int_{v_1>0}\int^1_0\mathcal{R}(y,v)dydv,\cr
		\mathcal{S}_R(f)&=\frac{1-\delta_1}{2-\delta_1}+\frac{\delta_1}{2-\delta_1}\int_{v_1>0}f_L|v_1|dv
		-\frac{1}{\tau(2-\delta_1)}\int_{v_1<0}\int^1_0\mathcal{R}(y,v)dydv,
	\end{align*}
	with 
	\begin{align*}
	\mathcal{R}(f)(x,v)=\mathcal{M}_{\nu}(f)(x,v)-f(x,v).
	\end{align*}
\end{definition} 
\begin{theorem}\label{Main2}{\bf[Diffusive dominant case]}\newline
	\noindent$(1)$ $($Non-critical $\nu$$)$ Let $-1/2<\nu<1$. Suppose $f_{LR}$ satisfies $(P_1)$.
	Then there exist  constants $K_1>0$, $\epsilon>0$  such that, if  $\tau>K_1$ and $\delta_1<\epsilon$,
	then there exists a unique mild solution $f\geq0$ to the boundary value problem  (\ref{ESBGK}), (\ref{boundary conditions}) and (\ref{flux c}) satisfying
	\begin{align*}
		\int_{\mathbb{R}^3}f(x,v)dv\geq a_{\ell,2},\quad \int_{\mathbb{R}^3}f(x,v)(1+|v|^2)dv\leq2C_{LR,2},
	\end{align*}
	and
	
	\begin{align*}
		C^1_{\nu}\delta_2^2\frac{\gamma_{\ell,2}}{27C_{LR,2}^2}\leq\kappa^{\top}\left\{\mathcal{T}_{\nu}\right\}\kappa\leq \frac{2 }{{3a_{\ell,2}}}C^2_{\nu}C_{LR,2},
	\end{align*}
where $\gamma_{\ell,2}$ denotes
\begin{align}\label{gamma dff}
\gamma_{\ell,2}=\left(\int_{v_1>0}e^{-\frac{1}{|v_1|}}M_w(0)|v_1|dv\right)
\left(\int_{v_1<0}e^{-\frac{1}{|v_1|}}M_w(1)|v_1|dv\right)>0.
\end{align}
	\noindent$(2)$ $($Critical $\nu$$)$ Let $\nu=-1/2$:
	Suppose $f_{LR}$ satisfies $(P_1)$. 
	Then there exists constants $K_1>0$, $\varepsilon_1, \varepsilon_2>0$ such that, if  $\tau>K_1$, $\delta_1<\epsilon_1$ 
	then there exists a unique mild solution $f\geq0$ to the boundary value problem  (\ref{ESBGK}), (\ref{boundary conditions}) and (\ref{flux c}) satisfying
	\begin{align*}
		\int_{\mathbb{R}^3}f(x,v)dv\geq a_{\ell,2},\quad \int_{\mathbb{R}^3}f(x,v)|v|^2dv\leq 
		2C_{LR,2},
	\end{align*}
	and
	\begin{align*}
		\delta_2\frac{a_{-1/2,2}}{4C_{LR,2}}\leq\kappa^{\top}\left\{\mathcal{T}_{-1/2}\right\}\kappa\leq \frac{3}{2a_{\ell,2}}C_{LR,2},\cr
	\end{align*}
	where $ a_{-1/2,2}$ denotes
	\begin{align}\label{a2}
a_{-1/2,2}=\inf_{|\kappa|=1}
	\int_{\mathbb{R}^3}e^{-\frac{1}{|v_1|}
	}M_w\left\{|v|^2-(v\cdot\kappa)^2\right\}dv>0.
	\end{align}
 \end{theorem}
\begin{remark}
Note that we don't impose any smallness restriction on the discrepancy of the boundary flux for the critical case ($\nu=-1/2$) in the diffusive dominant case. 
\end{remark}
Among others, the main difficulty comes from the derivation of the lower bound estimate of the temperature tensor $\mathcal{T}_{\nu}$. In the non-critical case $-1/2<\nu<1$, the temperature tensor satisfies the following equivalence relation:
\begin{equation}\label{ER}
\min\{1-\nu,1+2\nu\}T \mathbb{I}_3\leq\mathcal{T}_{\nu}\leq \max\{1-\nu,1+2\nu\}T\mathbb{I}_3.
\end{equation}
Therefore, it suffices to study the local temperature $T$, that can be shown to be bounded below by a quantity constructed from the boundary data.
In the critical case, $\nu=-1/2$, however, the first inequality of (\ref{ER}) becomes trivial, giving no information on the strict positivity of the temperature tensor.  
Our main observation in this case is that the bulk velocity in $x$ direction can be controlled by the discrepancy of the boundary flux even without the smallness of $\delta_1$ in the inflow dominance case:
\begin{align}\label{U1}
	\Big|U_1(x)\Big|
	&\leq\delta_1\left|\int_{v_1>0}f_L|v_1|dv-\int_{v_1<0}f_R|v_1|dv\right|
	+ O\!\left(\delta_2,\delta_3,\tau^{-1}\right),
\end{align}
which is physically relavent in that, if we don't have enough flux from both ends of the slab, we cannot expect fast flow inbetween.
We then observe that the quadratic polynomial of $\mathcal{T}_{-1/2}$ can be expressed using the local temperature and the directional temperature in the critical case: 
\begin{align*}
	\kappa^{\top}\left\{\mathcal{T}_{-1/2}\right\}\kappa
	&=\frac{1}{\rho}\int_{\mathbb{R}^3}f|v-U|^2dv-\frac{1}{\rho}\int_{\rrr}f\big\{(v-U)\cdot\kappa\big\}^2dv.
\end{align*}
We mention that the concept of "directional temperature" was coined by Villani in \cite{V2}, and  was crucially used in the proof of entropy production estimates of the Boltzmann equation. 
This, with the use of  (\ref{U1}), enables one to bound the temperature tensor in the critical case from below by a quantity defined only through the inflow boundary data and the inflow boundary flux:
\begin{align*}
&\frac{1}{2}\inf_{|\kappa|=1}\int_{\mathbb{R}^3}e^{-\frac{2C_{LM,1}}{|v_1|}
	}f_{LR}\left\{|v|^2-(v\cdot\kappa)^2\right\}dv
	-2\left|\int_{v_1>0}f_L|v_1|dv-\int_{v_1<0}f_R|v_1|dv\right|^2,
\end{align*}
up to small error. The first term can roughly be interpreted as the difference of total energy minus the directional energy of the inflow boundary data away from zero, and the second term is the discrepancy of the flux at both ends.
This enables one to bound the temperature tensor from below when $\delta_i\,(i=2,3)$ and $\tau^{-1}$ are sufficiently small.  

In the diffusive dominant case, similar argument is working but there is an important difference to be mentioned that, without additional assumption on the amount of flux given in (\ref{flux c}), we cannot expect the uniqueness of the solutions. Consider the following simple boundary value problem with diffusive boundary condition:
\begin{align*}
v_1\partial_xf=0,\qquad f(i,v)=\left(\int_{(-1)^{i+1}v_1>0}f(i,v)|v_1|dv\right)M_w(i),\quad(i=0,1).
\end{align*}
It can be easily checked that $$f(x,v)=C_1 M_w(v,0)1_{v_1>0}+C_2 M_w(v,1)1_{v_1<0}$$ solves the problem for
any $C_1,C_2>0$. 
In this regards, we impose the flux control condition (\ref{flux c}) in this case.

\subsection{Literature review}
 We start with the results on the stationary problems of the BGK model in a slab, which is most relevent to the current work.
 The first existence theory for stationary BGK model can be found in \cite{Ukai}, where Ukai applied the a version of Schauder fixed point theorem to solve the slab problem with inflow boundary condition.  
 In \cite{Nouri}, Nouri derived the existence of weak solutions for a quantum BGK model with a discretized condensation ansatz in a bounded interval.  In \cite{Bang Y}, classical Banach fixed point argument was developed to study the existence and uniqueness for slab problems for ES-BGK model. In \cite{Bang Y}, however, the boundary condition was limited to inflow boundary condition, and the case $\nu=-1/2$ is not treated, which is the main motivation of the current work. The argument of \cite{Bang Y} was then applied to a relativistic BGK model \cite{HY} and to the quantum BGK model \cite{BGCY}.

For the time dependent problems, it was Perthame who first obtained the existence of weak solutions  \cite{Perthame} under the assumption of finite mass, momentum, energy and entropy.
The unique mild solution was then found in \cite{P-P} in a function space with sufficient decay in the velocity domain.
Mischler extended this to  the whole space in \cite{Mischler}. Zhang et al considered the $L^p$ weak solution of the BGK model in \cite{Z-H}.
For the asymptotic stability near global equilibriums, we refer to \cite{Bello,Yun1}.
Various macroscopic limit for the BGK type models, including the hydrodynamic limit at the Euler and Navier-Stokes limit, Diffusion limit, and fractional limit can be found in \cite{DMOS,M-M-M,Mellet,SR1,SR2}.
For the development or analysis of numerical schemes for BGK models, see \cite{F-J,F-R,Issau,M,M-S,RSY,RY} and rich references therein.

As was mentioned in the introduction, after the verification of H-theorem of ES-BGK model made in \cite{ALPP}, the ES-BGK model got popularized a lot \cite{ABLP,F-J,G-T,M-S,Z-Stru}. Brull et al developed
a systematical way to derive of ES-BGK model and provided another proof of H-theorem in \cite{B3}. 
The entropy production estimate for ES-BGK model was obtained in \cite{Yun4}. 
For existence results, we refer to \cite{PY1} for weak solutions, \cite{Yun2} for unique mild solution 
and \cite{Yun3} for the result in near-global-equilibrium regime. For related results for the ES-BGK model for polyatomic molecules, see \cite{PY2,PY3,PY4,Yun22}.\newline


This paper is organized as follows. In Section 2, we set up an approximation scheme and the solution space for the
inflow dominance case.
In Section 3, we show that, under appropriate assumptions,  the approximate solution stays in the solution space in each iteration. The lower bound estimate for the temperature tensor in the critical case ($\nu=-1/2$) is made. In Section 4, we prove the Cauchy estimate to complete the proof of Theorem \ref{Main1}. Section 5 is devoted to the proof of Theorem \ref{Main2}. Since many parts overlap with proof of Theorem \ref{Main1}, we focus on the difference of the argument.
%
%
%
%
%

%
%
%
%

%
%
%
%
\section{Approximation scheme and solution space for inflow dominant case}
In the following,  we aim to construct the solution $f^n$ for
(\ref{ESBGK}). According to (\ref{macroscopic field}), $\rho^n$,
$U^n$, $T^n$ and $\Theta^n $ represent the hyrodynamic quantities
associated to $f^n$.  The
approximate solution approximate scheme reads:
\[
f^{n}(x,v)=f^n(x,v) 1_{v_1>0}+ f^n(x,v)1_{v_1<0},
\]
where $f^n_+$ and $f^n_-$ are determined iteratively by
\begin{align}\label{f+}
	\begin{split}
		f^{n+1}(x,v)&=e^{-\frac{x}{\tau|v_1|}}f^{n+1}(0,v)
		+\frac{1}{\tau|v_1|}\int_{0}^{x} e^{-\frac{x-y}{\tau|v_1|}}\mathcal{M}_{\nu}(f^n)dy
		\quad\text{if $v_{1}>0$}
	\end{split}
\end{align}
and
\begin{align}\label{f-}
	\begin{split}
		f^{n+1}(x,v)&=e^{-\frac{x}{\tau|v_1|}}f^{n+1}(1,v)
		+\frac{1}{\tau|v_1|}\int_{x}^1 e^{-\frac{1}{\tau|v_1|}}
		\mathcal{M}_{\nu}(f^n)dy
		\quad\text{if $v_{1}<0$}
	\end{split}
\end{align}
where $f^{n+1}(0,v)$ and $f^{n+1}(1,v)$ are defined by
\begin{align}\label{f_delta}
	\begin{split}
		f^{n+1}(0,v)&=\delta_1f_L(v)+\delta_2\left(\int_{v_1<0}f^n(0,v)|v_1|dv\right)M_w+\delta_3 f^n(0,Rv),\quad (v_1>0),\cr
		f^{n+1}(1,v)&=\delta_1f_R(v)+\delta_2\left(\int_{v_1>0}f^n(1,v)|v_1|dv\right)M_w+\delta_3 f^n(1,Rv),\quad (v_1>0).
	\end{split}
\end{align}
We will show that $\{f^n\}_n$ constructed from the above scheme satisfies several uniform-in-$n$ estimates. To do this in a more systematical way, we define two solution spaces. First we define the following solution space for the non-critical case $(-1/2<\nu<1)$:
\begin{align*}
\Omega_1=\Big\{f\in L^{\infty}\left([0,1]; L^1_2(\mathbb{R}^3_v)\right)\cap L^1_{\gamma,\langle v\rangle}(\mathbb{R}^3_v)~|&~f  \mbox{ satisfies } (\mathcal{A}_1), (\mathcal{B}_1), (\mathcal{C}_1), (\mathcal{D}_1)
\Big\}
\end{align*}
where $(\mathcal{A}_1)$, $(\mathcal{B}_1)$, $(\mathcal{C}_1)$ and $(\mathcal{D}_1)$ denote
\begin{itemize}
\item ($\mathcal{A}_1$) $f$ is non-negative:
\[
f(x,v)\geq0 \mbox{ for }x,v\in [0,1]\times \mathbb{R}^3.
\]
\item ($\mathcal{B}_1$) The macroscopic field is well-defined:
\begin{align*}
&\int_{\mathbb{R}^3}f(x,v)dv\geq a_{\ell,1},\quad \int_{\mathbb{R}^3}f(x,v)(1+|v|^2)dv\leq 
2C_{LR,1}.
\end{align*}
\item ($\mathcal{C}_1$) The temperature tensor is well-defined:
\begin{align*}
	C^1_{\nu}\delta_1^2\frac{\gamma_{\ell,1}}{3C_{LR,1}^2}\leq\kappa^{\top}\left\{\mathcal{T}_{\nu}\right\}\kappa\leq \frac{2 }{{3a_{\ell,1}}}C^2_{\nu}C_{LR,1}.
	\end{align*}
\item ($\mathcal{D}_1$) The trace is well-defined:
\begin{align*}
\|f\|_{L^1_{\gamma,|v_1|,\pm}}\leq 2\|f_{LR}\|_{L^1_{\gamma,|v_1|}},\quad \|f\|_{L^1_{\gamma,\langle v\rangle,\pm}}\leq 2C_{LR,1}.
\end{align*}
\end{itemize}
For the critical case $\nu=-1/2$, we define 
\begin{align*}
\Omega_2=\Big\{f\in L^{\infty}\left([0,1]; L^1_2(\mathbb{R}^3_v)\right)\cap L^1_{\gamma,\langle v\rangle}(\mathbb{R}^3_v)~|&~f  \mbox{ satisfies } (\mathcal{A}_2), (\mathcal{B}_2), (\mathcal{C}_2), (\mathcal{D}_2)
\Big\}
\end{align*}
 where $(\mathcal{A}_2)$, $(\mathcal{B}_2)$,  $(\mathcal{C}_2)$ and $(\mathcal{D}_2)$ denote
\begin{itemize}
\item ($\mathcal{A}_2$) $f$ is non-negative:
\[
f(x,v)\geq0 \mbox{ for }x,v\in [0,1]\times \mathbb{R}^3.
\]
\item ($\mathcal{B}_2$) The macroscopic field is well-defined:
\begin{align*}
&\int_{\mathbb{R}^3}f(x,v)dv\geq a_{\ell,1},\quad \int_{\mathbb{R}^3}f(x,v)(1+|v|^2)dv\leq 2C_{LR}.
\end{align*}
\item ($\mathcal{C}_2$) The temperature tensor is well-defined:
\begin{align*}
	\delta_1\frac{a_{-1/2}}{2C_{LR,1}}\leq\kappa^{\top}\left\{\mathcal{T}_{-1/2}\right\}\kappa\leq \frac{3}{2a_{\ell,1}}C_{LR,1}.
\end{align*}
\item ($\mathcal{D}_2$) The trace satisfies:
\begin{align*}
\|f\|_{L^1_{\gamma,|v_1|,\pm}}\leq 2\|f_{LR}\|_{L^1_{\gamma,|v_1|}},\quad \|f\|_{L^1_{\gamma,\langle v\rangle,\pm}}\leq 2C_{LR,1}.
\end{align*}
\end{itemize}

%
%
%
%

Before we move on to the proof of uniform estimates for $f^n$, we record a few estimates that will be fruitfully used throughout the paper.
\begin{lemma}\label{Max decomposition} $(1)$ Let $f\in\Omega_1$. Then there exists  positive constants $C$  depending only on the quantities (\ref{CLR}), (\ref{quantities}) and $\gamma_{\ell,1}$ such that
\[
\mathcal{M}_{\nu}(f)\leq  Ce^{-C|v|^2}.
\]
$(2)$ Let $f\in\Omega_2$. Then there exists  positive constants $C$  depending only on the quantities (\ref{CLR}), (\ref{quantities}) and $a_{-1/2}$ such that
\[
\mathcal{M}_{-1/2}(f)\leq  Ce^{-C|v|^2}.
\]
\end{lemma}
\begin{proof}
We only consider the proof of (2) to avoid repetition. We first note that the macroscopic velocity is well-defined in $\Omega_2$:
\begin{equation}\label{i1}
|U|=\frac{|\rho U|}{\rho}=\frac{\Big|\int_{\mathbb{R}^3}fvdv\Big|}{\int_{\mathbb{R}^3}fdv}\leq \frac{C_{LR}}{a_{\ell,1}}.
\end{equation}
On the other hand, $(\mathcal{C}_2)$ implies that
\begin{equation}\label{i2}
-(v-U)^{\top}\{\mathcal{T}\}^{-1}_{-1/2}(v-U)\leq - 
\frac{3}{2a_{\ell,1}}C_{LR,1}|v-U|^2,
\end{equation}
and
\begin{equation}\label{i3}
\det{\mathcal{T}_{-1/2}}=\lambda_1\lambda_2\lambda_3\geq 
\left\{\delta_1\frac{a_{-1/2,1}}{2C_{LR,1}}\right\}^3,
\end{equation}
where  $\lambda_i$ ($i=1,2,3$) to be the eigenvalues of $\mathcal{T}_{\nu}$. Note that $\mathcal{T}_{\nu}$
is diagonalizable since it's symmetric. The desired then estimate follows immediately from (\ref{i1}), (\ref{i2})
and (\ref{i3}). 
\end{proof}
%
%
%
%
The following lemma can be found in \cite{Bang Y}. We present the detailed proof for the readers' convenience.
\begin{lemma}\label{main estimate} Let $C$ be a fixed positive constants. Then we have
\begin{align*}
\int^x_0\int_{v_1>0}\frac{1}{\tau|v_1|}e^{-\frac{x-y}{\tau|v_1|}}e^{-Cv_1^2}dv_1dy\leq C\left(\frac{\ln \tau+1}{\tau}\right),\hspace{1cm} x\in [0,1]
\end{align*}
where $C>0$ depends only on quantities in (\ref{CLR}) and (\ref{quantities}).
\end{lemma}
First, we divide the domain of integration as follows:
\begin{eqnarray*}
\left\{\int^x_0\int_{|v_1|<\frac{1}{\tau}}+\int^x_0\int_{\frac{1}{\tau}\leq|v_1|<\tau}+\int^x_0\int_{|v_1|\geq\tau}\right\}\frac{1}{\tau|v_1|}e^{-\frac{x-y}{\tau|v_1|}}e^{-Cv_1^2}dv_1dy
\equiv I_1+I_2+I_3.
\end{eqnarray*}
(a) The estimate of $I_1$:
For $I_1$, we integrate on $y$ first to get
\begin{align*}
I_1&=\int_{|v_1|<\frac{1}{\tau}}\left\{\int_{0}^x\frac{1}{\tau|v_1|}e^{-\frac{x-y}{\tau|v_1|}}dy\right\}e^{-Cv_1^2}dv_1\cr
&=\frac{1}{a_{\ell,1}}\int_{|v_1|<\frac{1}{\tau}}\left\{1-e^{-\frac{x}{\tau|v_1|}}\right\}e^{-Cv_1^2}dv_1\cr
&\leq \int_{|v_1|<\frac{1}{\tau}}dv_1\cr
&\leq \frac{1}{\tau} .
\end{align*}

\noindent(b) The estimate of $I_2$: For this case, we find
\begin{align*}
	I_2
	&\leq\frac{1}{a_{\ell,1}}\int_{\frac{1}{\tau}\leq|v_1|\leq\tau}1-e^{-\frac{a_{\ell,1}x}{\tau|v_1|}}\,dv_1,
\end{align*}
and apply the Taylor expansion to $1-e^{-\frac{1}{\tau|v_1|}}$ to get
\begin{align*}
I_{2}&\leq\int_{\frac{1}{\tau}<|v_1|<\tau}
\left\{\left(\frac{1}{\tau|v_1|}\right)-\frac{1}{2!}\left(\frac{1}{\tau|v_1|}\right)^2+\frac{1}{3!}\left(\frac{1}{\tau|v_1|}\right)^3+\cdots\right\}dv_1\cr
&\leq\Big|\int_{\frac{1}{\tau}}^{\tau}\frac{1}{\tau r}\,dr\Big|+\Big|\int_{\frac{1}{\tau}}^{\tau}\frac{1}{2!}\left(\frac{1}{\tau r}\right)^2dr\Big|+
\Big|\int_{\frac{1}{\tau}}^{\tau}\frac{1}{3!}\left(\frac{1}{\tau r}\right)^3dr \Big|+\cdots\cr
&=\frac{1}{\tau}\ln \tau^2+\frac{1}{2!}\frac{1}{\tau}\frac{\tau^2-1}{\tau^2}
+\frac{1}{2\cdot 3!}\frac{1}{\tau}\frac{\tau^4-1}{\tau^4}+
\frac{1}{3\cdot 4!}\frac{1}{\tau}\frac{\tau^6-1}{\tau^6}
\cdots\cr
&\leq \frac{1}{\tau}\ln \tau^2+\frac{e}{\tau}.
\end{align*}
\noindent$(c)$ We compute the remaining $I_3$ as
\begin{eqnarray*}
I_3
\leq\int_{0}^1\int_{|v_1|>\tau}\frac{1}{\tau|v_1|}e^{-Cv_1^2}dv_1dy
\leq\frac{1}{\tau^2}\int_{|v_1|>\tau}e^{-Cv_1^2}dv_1
\leq C_{\ell,u}\frac{1}{\tau^2}.
\end{eqnarray*}
Finally, we combine the estimates $(a)$, $(b)$, $(c)$ to obtain the desired result.
%
%
%
%
\section{ Uniform-in-$n$ estimates of $f^n$ for inflow dominant case}
The main result of this section is the following proposition
\begin{proposition}\label{fixed point 1}
\noindent$(1)$ Let $-1/2<\nu<1$. Assume $f_{LR}$ satisfies the conditions of Theorem 1.1 (1). Then $f^{n}\in \Omega_1$ for all $n$.\newline
\noindent$(2)$ Let $\nu=-1/2$. Assume $f_{LR}$ satisfies the conditions of Theorem 1.1 (2). Then, $f^{n}\in \Omega_2$ for all $n$.
\end{proposition}
We divide the proof into Lemma \ref{al}, Lemma \ref{au}, and Lemma \ref{gl}.
\begin{lemma}\label{al} Assume $f^n\in \Omega_1$ or $\Omega_2$. Then we have
\begin{align*}
f^{n+1}\geq0\quad\mbox{ and }\quad \int_{\mathbb{R}^3}f^{n+1}dv\geq \delta_1a_{\ell,1}.
\end{align*}
\end{lemma}
\begin{proof}
 From (\ref{mild f+}) and (\ref{mild f-}), we find
\begin{align*}
f^{n+1}&\geq \delta_1 e^{-\frac{x}{\tau|v_1|}}f_{L}(v)1_{v_1>0}
+\delta_1 e^{-\frac{x}{\tau|v_1|}}f_{R}(v)1_{v_1<0}\geq\delta_1e^{-\frac{C_{LR,1}}{|v_1|}}f_{LR}.
\end{align*}
Here, we assumed $\tau>1$ without loss of generality.
Integrating with respect to $dv$, we obtain the desired lower bound:
\begin{align*}
\int_{\mathbb{R}^3}f^{n+1}dv
\geq\delta_1\int_{\mathbb{R}^3}e^{-\frac{1 }{|v_1|}}f_{LR}dv
= \delta_1a_{\ell,1}.
\end{align*}

\end{proof}

\begin{lemma}\label{auu}
	\noindent$(1)$ Let $f^n\in \Omega_1$ or $\Omega_2$. Then we have
	\begin{align*}
	\|f^{n+1}\|_{L^1_{\gamma,|v_1|,+}},~ \|f^{n+1}\|_{L^1_{\gamma,|v_1|,-}}\leq 2\|f_{LR}\|_{L^1_{\gamma,|v_1|}}.
	\end{align*}
	$(2)$ Let $f^n\in \Omega_1$ or $\Omega_2$. Then we have
	\begin{align*}
	\|f^{n+1}\|_{L^1_{\gamma,\langle v\rangle,+ }},~ \|f^{n+1}\|_{L^1_{\gamma,\langle v\rangle,- }}\leq 2\|M_w\|_{L^1_{\gamma,\langle v\rangle }}\|f_{LR}\|_{L^1_{\gamma,\langle v\rangle }}.
	\end{align*}
\end{lemma}
\begin{proof} (1) $\bullet$ Estimate for outflux $\|f^{n+1}\|_{L^1_{\gamma,|v_1|,+}}$: 
	Using (\ref{f-}), we can write $f^{n+1}(0,v)$ for $v_1<0$ as
	\begin{align}\label{int2}
	\begin{split}
	f^{n+1}(0,v)&=\delta_1 f_R+\delta_2\left(\int_{v_1>0}f(1,v)|v_1|dv\right)M_w(1)+\delta_3 f^n(1,Rv)\cr
	&+\frac{1}{\tau|v_1|}\int_{0}^1 e^{-\frac{y}{\tau|v_1|}}
	\mathcal{M}_{\nu}(f^n)dy
	\end{split}
	\end{align}
	which, in view of  Lemma \ref{main estimate}, yields
	\begin{align}\label{f22}
	\begin{split}
	&\int_{v_1<0}f^{n+1}(0,v)|v_1|dv\cr
	&\leq \delta_1\int_{v_1<0}f_R|v_1|dv+\delta_2\int_{v_1>0}f^{n}(1,v)|v_1|dv+\delta_3\int_{v_1>0}f^{n}(1,v)|v_1|dv\cr
	&+C_{\ell,u}\left(\frac{\ln \tau+1}{\tau}\right).
	\end{split}
	\end{align}
	Similarly,
		\begin{align}\label{f23}
	\begin{split}
	&\int_{v_1>0}f^{n+1}(1,v)|v_1|dv\cr
	&\leq \delta_1\int_{v_1>0}f_L|v_1|dv+\delta_2\int_{v_1<0}f^{n}(0,v)|v_1|dv+\delta_3\int_{v_1<0}f^{n}(0,v)|v_1|dv\cr
	&+C_{\ell,u}\left(\frac{\ln \tau+1}{\tau}\right).
	\end{split}
	\end{align}
	From (\ref{f22}) and $(\ref{f23})$, we obtain
	\begin{align}\label{ff11}
	\begin{split}
	\|f^{n+1}\|_{L^1_{\gamma,|v_1|,+}}&=\int_{v_1<0}f^{n+1}(0,v)|v_1|dv+\int_{v_1>0}f^{n+1}(1,v)|v_1|dv\cr
	&\leq \delta_1\|f_{LR}\|_{L^1_{\gamma,|v_1|}}+(\delta_2+\delta_3)\|f^n\|_{L^1_{\gamma,|v_1|,+}}\cr
	&+C_{\ell,u}\left(\frac{\ln \tau+1}{\tau}\right).
	\end{split}
	\end{align}
	
	Therefore, in view of $\mathcal{D}_i $ of $\Omega_i$ $(i=1,2)$, we see that
	\begin{align}\label{outflux}
	\begin{split}
	\|f^{n+1}\|_{L^1_{\gamma,|v_1|,+}}
	&\leq
	\delta_1\|f_{LR}\|_{L^1_{\gamma,|v_1|}}
	+2(\delta_2+\delta_3)\|f_{LR}\|_{L^1_{\gamma,|v_1|}}
	+C_{\ell,u}\left(\frac{\ln \tau+1}{\tau}\right)\cr
	&=2(\delta_1+\delta_2+\delta_3)\|f_{LR}\|_{L^1_{\gamma,|v_1|}}-\delta_1\|f_{LR}\|_{L^1_{\gamma,|v_1|}}
	+C_{\ell,u}\left(\frac{\ln \tau+1}{\tau}\right)\cr
	&\leq 2\|f_{LR}\|_{L^1_{\gamma,|v_1|}}
	\end{split}
	\end{align}
	for sufficiently large $\tau$.\newline

	\noindent $\bullet$ Estimate for influx: $\|f^{n+1}\|_{L^1_{\gamma,|v_1|,-}}$: When $v_1>0$, we have from the boundary condition (\ref{f+}) that
	\begin{align}\label{int1}
	f^{n+1}(0,v)=\delta_1 f_L+\delta_2\left(\int_{v_1<0}f^n(0,v)|v_1|dv\right)M_w(0)+\delta_3 f^n(0,Rv).
	\end{align}
	Integrate both sides with respect to $|v_1|dv$ on $v_1>0$ to get
	\begin{align}\label{f111}
	\begin{split}
	&\int_{v_1>0}f^{n+1}(0,v)|v_1|dv\cr
	&\quad= \delta_1\int_{v_1>0}f_L|v_1|dv+\delta_2\left(\int_{v_1>0}M_w(0)|v_1|dv\right)\left(\int_{v_1<0}f^{n}(0,v)|v_1|dv\right)\cr
	&\quad+\delta_3\int_{v_1>0}f^{n}(0,Rv)|v_1|dv\cr
	&\quad\leq \delta_1\int_{v_1>0}f_L|v_1|dv+(\delta_2+\delta_3)\int_{v_1<0}f^{n}(0,v)|v_1|dv,
	\end{split}
	\end{align}
	where we used $\int_{v_1>0}M_w(0)|v_1|dv=1$.
	Similarly, we estimate 	
	\begin{align}\label{f112}
	\begin{split}
	&\int_{v_1<0}f^{n+1}(1,v)|v_1|dv\cr
	&\quad= \delta_1\int_{v_1<0}f_R|v_1|dv+\delta_2\left(\int_{v_1<0}M_w(1)|v_1|dv\right)\left(\int_{v_1>0}f^{n}(1,v)|v_1|dv\right)\cr
	&\quad+\delta_3\int_{v_1>0}f^{n}(1,Rv)|v_1|dv\cr
	&\quad\leq \delta_1\int_{v_1<0}f_R|v_1|dv+(\delta_2+\delta_3)\int_{v_1>0}f^{n}(1,v)|v_1|dv,
	\end{split}
	\end{align}
	Combining (\ref{f111}) and (\ref{f112}) gives
	\begin{align}\label{ff11-}
	\begin{split}
	\|f^{n+1}\|_{L^1_{\gamma,|v_1|,-}}&=\int_{v_1>0}f^{n+1}(0,v)|v_1|dv+\int_{v_1<0}f^{n+1}(1,v)|v_1|dv\cr
	&\leq \delta_1\|f_{LR}\|_{L^1_{\gamma,|v_1|}}+(\delta_2+\delta_3)\|f^n\|_{L^1_{\gamma,|v_1|,+}}.
	\end{split}
	\end{align}
	Thanks to (\ref{outflux}), we have
	\begin{align*}
	\begin{split}
	\|f^{n+1}\|_{L^1_{\gamma,|v_1|,-}}&\leq\delta_1\|f_{LR}\|_{L^1_{\gamma,|v_1|}}+2(\delta_2+\delta_3)
	\|f_{LR}\|_{L^1_{\gamma,|v_1|}}\cr
	&\leq 2\|f_{LR}\|_{L^1_{\gamma,|v_1|}}.
	\end{split}
	\end{align*}
	This completes the proof of (1). The proof of (2) is identical. We omit the proof.
\end{proof}

\begin{lemma}\label{au}  Let $f^n\in \Omega_1$ or $\Omega_2$. Then we have
\begin{align*}
\int_{\mathbb{R}^3}f^{n+1}(1+|v|^2)dv\leq 4\|f_{LR}\|_{L^1_{\gamma,\langle v\rangle}}\|M_w\|_{L^1_{\gamma,\langle v\rangle}}.
\end{align*}
\end{lemma}
\begin{proof}
We integrate (\ref{f+}) w.r.t $(1+|v|^2)dv$ on $v_1>0$ to get
\begin{align*}
\int_{v_1>0}f^{n+1}(x,v)(1+|v|^2)dv&=\int_{v_1>0}e^{-\frac{x}{\tau|v_1|}}f^{n+1}(0,v)(1+|v|^2)dv\cr
&+\int_{v_1>0}\int_{0}^{x}\frac{1}{\tau|v_1|}e^{-\frac{x-y}{\tau|v_1|}}\rho^n(y)\mathcal{M}_{\nu}(f_{n})(1+|v|^2)dydv.
\end{align*}
We compute the boundary terms in $\int_{\mathbb{R}^3}f^{n+1}(1+|v|^2)dv$ as
\begin{align*}
&\int_{v_1>0}e^{-\frac{x}{\tau|v_1|}}f^{n+1}(0,v)(1+|v|^2)dv\cr
&\qquad\leq\int_{v_1>0}f^{n+1}(0,v)(1+|v|^2)dv\cr
&\qquad\leq \delta_1\int_{v_1>0}f_L(v)(1+|v|^2)dv+\delta_2\left(\int_{v_1<0}f^n(0,v)|v_1|dv\right)\|M_w\|_{L^1_{\gamma,\langle v\rangle,-}}\cr
&\qquad+\delta_3\int_{v_1<0}f^n(0,v)(1+|v|^2)dv.
\end{align*}
Thanks to Lemma \ref{Max decomposition}, we can compute the source term as
\begin{align*}
&\int_{v_1>0}\int_{0}^{x}\frac{1}{\tau|v_1|}e^{-\frac{x-y}{\tau|v_1|}}\mathcal{M}_{\nu}(f)(1+|v|^2)dydv\cr
&\qquad\leq C_{LM,1}\int_{0}^{x}\int_{v_1>0} \frac{1}{\tau|v_1|}e^{-\frac{x-y}{\tau|v_1|}}e^{-Cv^2_1}dv_1dy\cr
&\qquad\equiv C_{LM,1}  \left(  \frac{\ln \tau +1  }{\tau}  \right).
\end{align*}
Therefore,
\begin{align*}
&\int_{v_1>0}f^{n+1}(x,v)(1+|v|^2)dv\cr
&\qquad= \delta_1\int_{v_1>0}f_L(v)(1+|v|^2)dv+\delta_2\left(\int_{v_1<0}f^n(0,v)|v_1|dv\right)\|M_w\|_{L^1_{\gamma,\langle v\rangle,-}}\cr
&\qquad+\delta_3\int_{v_1<0}f^n(0,v)(1+|v|^2)dv+C_{\ell,u}\left(\frac{\ln \tau+1}{\tau}\right).
\end{align*}
We can have similar result for the case $v_1<0$:
\begin{align*}
&\int_{v_1<0}f^{n+1}(x,v)(1+|v|^2)dv\cr
&\qquad= \delta_1\int_{v_1<0}f_R(v)(1+|v|^2)dv+\delta_2\left(\int_{v_1>0}f^n(1,v)|v_1|dv\right)\|M_w\|_{L^1_{\gamma,\langle v\rangle,+}}\cr
&\qquad+\delta_3\int_{v_1>0}f^n(1,v)(1+|v|^2)dv+C_{\ell,u}\left(\frac{\ln \tau+1}{\tau}\right).
\end{align*}
Summing up and using $\mathcal{D}_i$ of $\Omega_i$ $(i=1,2)$, we get
\begin{align*}
&\int_{\mathbb{R}^3}f^{n+1}(x,v)(1+|v|^2)dv\cr
&\qquad\leq \delta_1\|f_{LR}\|_{L^1_{\gamma,\langle v\rangle}}+\delta_2\|M_w\|_{L^1_{\gamma,\langle v\rangle}}\|f^n\|_{L^1_{\gamma,|v_1|,+}}+\delta_3\|f^n\|_{L^1_{\gamma,\langle v\rangle,+}}+C_{\ell,u}\left(\frac{\ln \tau+1}{\tau}\right)\cr
&\qquad\leq \delta_1\|f_{LR}\|_{L^1_{\gamma,\langle v\rangle}}+\delta_2\|M_w\|_{L^1_{\gamma,\langle v\rangle}}\|f^n\|_{L^1_{\gamma,\langle v_1\rangle}}+2\delta_3\|f_{LR}\|_{L^1_{\gamma,\langle v\rangle}}+C_{\ell,u}\left(\frac{\ln \tau+1}{\tau}\right).
\end{align*}
Since 
\[
1=\|M_w\|_{L^1_{\gamma,|v_1|}}\leq \|M_w\|_{L^1_{\gamma,\langle v\rangle}},
\]
we have
\begin{align*}
&\int_{\mathbb{R}^3}f^{n+1}(x,v)(1+|v|^2)dv\cr
&\qquad\leq \delta_1\|f_{LR}\|_{L^1_{\gamma,\langle v\rangle}}\|M_w\|_{L^1_{\gamma,\langle v\rangle}}
+\delta_2\|f^n\|_{L^1_{\gamma,\langle v_1\rangle}}\|M_w\|_{L^1_{\gamma,\langle v\rangle}}+2\delta_3\|f_{LR}\|_{L^1_{\gamma,\langle v\rangle}}\|M_w\|_{L^1_{\gamma,\langle v\rangle}}\cr
&\qquad+C_{\ell,u}\left(\frac{\ln \tau+1}{\tau}\right)\cr
&\qquad\leq 2(\delta_1+\delta_2+\delta_3)\|f^n\|_{L^1_{\gamma,\langle v_1\rangle}}\|M_w\|_{L^1_{\gamma,\langle v\rangle}}-(\delta_1+\delta_3)\|f_{LR}\|_{L^1_{\gamma,\langle v\rangle}}\|M_w\|_{L^1_{\gamma,\langle v\rangle}}\cr
&\qquad+C_{\ell,u}\left(\frac{\ln \tau+1}{\tau}\right)\cr
&\qquad\leq 2\|f_{LR}\|_{L^1_{\gamma,\langle v\rangle}}\|M_w\|_{L^1_{\gamma,\langle v\rangle}}.
\end{align*}
\end{proof}
In the following lemma, we show that (1) the bulk velocity in $x$ direction is controlled by the discrepancy of the boundary flux, and (2) the bulk velocities in other directions can be taken arbitrarily small.
\begin{lemma}\label{23}
Let $f^n\in\Omega_1$ or $\Omega_2$. \newline
\noindent$(1)$ For $i=1$, we have
\begin{align*}
\Big|\int_{\mathbb{R}^3}f^{n+1}v_1dv\Big|
&\leq\delta_1\left|\int_{v_1>0}f_L|v_1|dv-\int_{v_1<0}f_R|v_1|dv\right|
+2\left(\delta_2+\delta_3+\frac{2}{\tau}\right)C_{LR,1}
\end{align*}
where $C_{LR,1}$ denotes
\[
C_{LR,1}=\|f_{LR}\|_{L^1_{\gamma,\langle v\rangle}}\|M_w\|_{L^1_{\gamma,\langle v\rangle}}.
\]
\noindent$(2)$ For $i=2,3$, we have
\begin{align*}
\Big|\int_{\mathbb{R}^3}f^{n+1}v_idv\Big|\leq
2\delta_3\|f_{LR}\|_{L^1_{\gamma,|v_1|}}
\|M_w\|_{L^1_{\gamma,|v_1|}}+C_{\ell,u}\left(\frac{\ln \tau+1}{\tau}\right).
\end{align*}
\end{lemma}
\begin{proof}
\noindent$(1)$
For $v_1>0$, (\ref{f+}) has the equivalent mild form:
\begin{align*}
f^{n+1}(x,v)&=f^{n+1}(0,v)+\frac{1}{\tau|v_1|}\int^x_0\left(\mathcal{M}_{\nu}(f^n)-f^{n+1}\right)dy\cr
&=f^{n+1}(0,v)+\frac{1}{\tau v_1}\int^x_0\left(\mathcal{M}_{\nu}(f^n)-f^{n+1}\right)dy.
\end{align*}
We multiply $v_1$ and integrate on $v_1>0$ to get
\begin{align}\label{(a)}
\int_{v_1>0}f^{n+1}(x,v) v_1dv&=\int_{v_1>0}f^{n+1}(0,v)|v_1|dv+\frac{1}{\tau}\int^x_0\int_{v_1>0}\left(\mathcal{M}_{\nu}(f^n)-f^{n+1}\right)dydv. 
\end{align}
In the case $v_1<0$, we have from (\ref{f-}) that
\begin{align*}
f^{n+1}(x,v)&=f^{n+1}(1,v)+\frac{1}{\tau|v_1|}\int^x_1\left(\mathcal{M}_{\nu}(f^n)-f^{n+1}\right)dy\cr
&=f^{n+1}(1,v)-\frac{1}{\tau v_1}\int^x_1\left(\mathcal{M}_{\nu}(f^n)-f^{n+1}\right)dy.
\end{align*}
Integrating with respect to $v_1dv$ on $v_1<0$:
\begin{align}\label{(b)}
\int_{v_1<0}f^{n+1}(x,v) v_1dv
&=\int_{v_1<0}f^{n+1}(1,v)v_1dv-\frac{1}{\tau}\int^x_1\int_{v_1<0}\left(\mathcal{M}_{\nu}(f^n)-f^{n+1}\right)dydv\cr
&=-\int_{v_1<0}f^{n+1}(1,v)|v_1|dv+\frac{1}{\tau}\int^1_x\int_{v_1<0}\left(\mathcal{M}_{\nu}(f^n)-f^{n+1}\right)dydv.
\end{align}
From (\ref{(a)}) and (\ref{(b)}), we have
\begin{align*}
\int_{\mathbb{R}^3}f^{n+1}(x,v) v_1dv&=I+II,\cr
\end{align*}
where
\begin{align*}
I
&=\int_{v_1>0}f^{n+1}(0,v)|v_1|dv-\int_{v_1<0}f^{n+1}(1,v)|v_1|dv
\end{align*}
and
\begin{align*}
II&=\frac{1}{\tau}\int^x_0\int_{v_1>0}\left(\mathcal{M}_{\nu}(f^n)-f^{n+1}\right)dvdy
+\frac{1}{\tau}\int^1_x\int_{v_1<0}\left(\mathcal{M}_{\nu}(f^n)-f^{n+1}\right)dydv.
\end{align*}
$(a)$ The estimate of $I$: 
We observe from the boundary condition that
\begin{align}\label{I1}
	\begin{split}
&\int_{v_1>0}f^{n+1}(0,v)|v_1|dv\cr
&\qquad=\delta_1\int_{v_1>0}f_L|v_1|dv
+\delta_2\left(\int_{v_1<0}f^n(0,v)|v_1|dv\right)\left(\int_{v_1>0}M_w|v_1|dv\right)\cr
&\qquad+\delta_3\int_{v_1>0}f^n(0,Rv)|v_1|dv\cr
&\qquad=\delta_1\int_{v_1>0}f_L|v_1|dv
+\delta_2\int_{v_1<0}f^n(0,v)|v_1|dv
+\delta_3\int_{v_1<0}f^n(0,v)|v_1|dv\cr
\end{split}
\end{align}
and 
\begin{align}\label{I2}
	\begin{split}
&\int_{v_1<0}f^{n+1}(1,v)|v_1|dv\cr
&\qquad=\delta_1\int_{v_1<0}f_R|v_1|dv
+\delta_2\int_{v_1>0}f^n(1,v)|v_1|dv+\delta_3\int_{v_1>0}f^n(1,v)|v_1|dv.
\end{split}
\end{align}
Taking the difference of (\ref{I1}) and (\ref{I2}), we estimate $I$ as
\begin{align*}
|I|&=\left|\int_{v_1>0}f^{n+1}(0,v)|v_1|dv-\int_{v_1<0}f^{n+1}(1,v)|v_1|dv\right|\cr
&\leq\delta_1\left|\int_{v_1>0}f_L|v_1|dv-\int_{v_1<0}f_R|v_1|dv\right|
+\delta_2\|f^n\|_{L^1_{\gamma,\langle v\rangle}}\|f_{LR}\|_{L^1_{\gamma,\langle v\rangle}}+\delta_3\|f^n\|_{L^1_{\gamma,\langle v\rangle}}\cr
&\leq\delta_1\left|\int_{v_1>0}f_L|v_1|dv-\int_{v_1<0}f_R|v_1|dv\right|
+2(\delta_2+\delta_3)\|f_{LR}\|_{L^1_{\gamma,\langle v\rangle}}\|M_w\|_{L^1_{\gamma,\langle v\rangle}}.
\end{align*}
In the last line, we used Lemma \ref{auu}.\newline

\noindent$(b)$ The estimate for $II$: For $II$, we compute
\begin{align*}
|II|&\leq \frac{1}{\tau}\int^x_0\int_{v_1>0}\left|\mathcal{M}_{\nu}(f^n)-f^{n+1}\right|dvdy
+\frac{1}{\tau}\int^1_x\int_{v_1<0}\left|\mathcal{M}_{\nu}(f^n)-f^{n+1}\right|dydv\cr
&\leq\frac{1}{\tau}\int^1_0\int_{v_1>0}\left\{\mathcal{M}_{\nu}(f^n)+f^{n+1}\right\}dvdy
+\frac{1}{\tau}\int^1_0\int_{v_1<0}\left\{\mathcal{M}_{\nu}(f^n)+f^{n+1}\right\}dydv\cr
&=\frac{1}{\tau}\int^1_0\int_{\mathbb{R}^3}\{\mathcal{M}_{\nu}(f^n)+f^{n+1}\}dvdy\cr
&=\frac{1}{\tau}\int^1_0\left\{\rho^{n}+\rho^{n+1}\right\}dy\cr
&\leq \frac{2}{\tau}\|f_{LR}\|_{L^1_{\gamma,\langle v\rangle}}\|M_w\|_{L^1_{\gamma,\langle v\rangle}}.
\end{align*}
In the last line, we used $(\mathcal{B}_i)$ $(i=1,2):$
\begin{align*}
\rho^n\leq 2\|f_{LR}\|_{L^1_{\gamma,\langle v\rangle}}\|M_w\|_{L^1_{\gamma,\langle v\rangle}}
\end{align*}
and Lemma \ref{au}:
\begin{align*}
\rho^{n+1}&\leq 2\|f_{LR}\|_{L^1_{\gamma,\langle v\rangle}}\|M_w\|_{L^1_{\gamma,\langle v\rangle}}.
\end{align*}
Now, we combine $(a)$ and $(b)$ to obtain the desired result.\newline\newline
\noindent (2) We only prove the case $i$=$2$.
For  $v_1>0$, we integrate (\ref{f+}) with respect to $v_2dv_2dv_3$ to get
\begin{align*}
\int_{\mathbb{R}^2}f^{n+1}(x,v)v_2dv_2dv_3&=e^{-\frac{x}{\tau |v_1|}}
\int_{\mathbb{R}^2}f^{n+1}(0,v)v_2dv_2dv_3\cr
&+\frac{1}{\tau|v_1|}\int^x_0e^{-\frac{x-y}{\tau |v_1|}}\int_{\mathbb{R}^2}\mathcal{M}_{\nu}(f^n)v_2dv_2dv_3dy
\end{align*}

From $(P_2)$ of $f_L$ that no vertical flow is induced from $f_{LR}$, and the fact that $M_wv_2$ is odd in $v_2$, we have
\begin{align*}
\int_{\mathbb{R}^2}f^{n+1}(0,v)v_2dv_2dv_3=\delta_3\int_{v_1<0}f^n(0,v)dv
\end{align*}
which, together with Property $(\mathcal{B}_1)$ or $(\mathcal{B}_2)$ and Lemma \ref{Max decomposition} gives
\begin{align*}
	\int_{\mathbb{R}^2}f^{n+1}(x,v)v_2dv_2dv_3&\leq \delta_3\int_{\mathbb{R}^2}f^{n}(0,v)v_2dv_2dv_3\cr
	&+C_{\ell,u}\frac{1}{\tau|v_1|}\int_{\mathbb{R}^6}\int^x_0e^{-\frac{x-y}{\tau |v_1|}}e^{-C|v|^2}dydv_2dv_3.
\end{align*}
Now, integrating on $v_1>0$ and recalling Lemma \ref{main estimate}, we get
\begin{align*}
\left|\int_{v_1>0}f^{n+1}(x,v)v_2dv\right|
\leq \delta_3\left|\int_{v_1<0}f^n(0,v)v_2dv\right|+C_{\ell,u}\left(\frac{\ln \tau+1}{\tau}\right).
\end{align*}
Applying the same type of argument to (\ref{f-}), we can derive
\begin{align*}
\Big|\int_{v_1<0}f^{n+1}(x,v)v_2dv\Big|
\leq \delta_3\left|\int_{v_1>0}f^n(1,v)v_2dv\right|+C_{\ell,u}\left(\frac{\ln \tau+1}{\tau}\right).
\end{align*}
We sum these up them to obtain
\begin{align*}
\Big|\int_{v_1<0}f^{n+1}(x,v)v_2dv\Big|
&\leq 2\delta_3\|f_{LR}\|_{L^1_{\gamma,\langle v_1\rangle }}\|M_w\|_{L^1_{\gamma,\langle v_1\rangle}}+C_{\ell,u}\left(\frac{\ln \tau+1}{\tau}\right).
\end{align*}
Here, we used Lemma \ref{auu} (2).
\end{proof}
In the following lemma, we show that the quadratic polynomial of the temperature tensor can be controlled from below and above. We mention that the critical case ($\nu=-1/2$) has never been
treated in the literature so far, except for the near-global-equilibrium regime \cite{Yun22,Yun3}.
\begin{lemma}\label{gl} $(1)$ Let $-1/2<\nu<1$. Assume $f^n\in\Omega_1$. Then, for sufficiently large $\tau$, we have
\begin{align*}
C^1_{\nu}\delta_1^2\frac{\gamma_{\ell,1}}{3C_{LR,1}^2}\leq\kappa^{\top}\left\{\mathcal{T}^{n+1}_{\nu}\right\}\kappa\leq \frac{2 }{{3a_{\ell,1}}\delta_1}C^2_{\nu}C_{LR,1}
\end{align*}
for any $\kappa\in\mathbb{R}$ and $|\kappa|=1$. \newline
\noindent$(2)$ Let $\nu=-1/2$ and $f\in \Omega_2$. Then, for sufficiently large $\tau$, we have
\begin{align*}
\delta_1\frac{a_{-1/2,1}}{2C_{LR,1}}\leq\kappa^{\top}\left\{\mathcal{T}^{n+1}_{-1/2}\right\}\kappa\leq \frac{3}{2a_{\ell,1}\delta_1}C_{LR,1}
\end{align*}
for any $\kappa\in\mathbb{R}$ and $|\kappa|=1$. 
\end{lemma}
\begin{remark}
	We recall that where $\gamma_{\ell,1}$ and $a_{-1/2,1}$ are defined in (\ref{gamma}) and (\ref{a1}) respectively.
\end{remark}
\begin{proof}
(1) We recall the following equivalence estimate \cite{Bang Y,Yun3} which holds for $-1/2<\nu<1$:
\begin{eqnarray}\label{equiv T}
C^1_{\nu}T^{n+1}\mathbb{I}_3\leq\mathcal{T}^{n+1}_{\nu}\leq C^2_{\nu}T^{n+1}\mathbb{I}_3,
\end{eqnarray}
where
$C^1_{\nu}=\min\{1-\nu,1+2\nu\}$ and $C^2_{\nu}=\max\{1-\nu,1+2\nu\}$.
Therefore, it is enough to derive the lower and upper bound of $T^{n+1}$. 
The upper bound follows easily from Lemma \ref{al} and Lemma \ref{au}:
\begin{align}\label{T1}
\begin{split}
T^{n+1}
\leq \frac{1}{3\rho^{n+1}}\int_{\mathbb{R}^3}f^{n+1}|v|^2dv
\leq \frac{2}{3\delta_1a_{\ell,1}}\|f_{LR}\|_{L^1_{\gamma,\langle v\rangle}}\|M_w\|_{L^1_{\gamma,\langle v\rangle}}.
\end{split}
\end{align}
Now we turn to the lower bound of $T^{n+1}$.
Since $f^{n+1}\geq0$ and $|v|\geq |v_1|$, we have from the Cauchy-Schwarz inequality that
\begin{align*}
3\{\rho^{n+1}\}^2T^{n+1}&=\left(\int_{\mathbb{R}^3}f^{n+1}dv\right)\left(\int_{\mathbb{R}^3}f^{n+1}|v|^2dv\right)
-\left|\int_{\mathbb{R}^3}f^{n+1}vdv\right|^2\cr
&\geq\left(\int_{\mathbb{R}^3}f^{n+1}|v_1|dv\right)^2-\left|\int_{\mathbb{R}^3}f^{n+1}vdv\right|^2.
\end{align*}
Then we decompose according to whether they contain vertical flow or not:
\begin{align}\label{I}
	\begin{split}
3\{\rho^{n+1}\}^2T^{n+1}
& \geq\left(\int_{\mathbb{R}^3}f^{n+1}|v_1|dv\right)^2-\Big\{\sum_{1\leq i\leq 3}\Big|\int_{\mathbb{R}^3}f^{n+1}v_idv\Big|\Big\}^2\cr
&=\left(\int_{\mathbb{R}^3}f^{n+1}|v_1|dv\right)^2
-\Big(\int_{\mathbb{R}^3}f^{n+1}v_1dv\Big)^2-R\cr
&\equiv I-R,
\end{split}
\end{align}
where
\begin{align*}
R=\sum_{(i,j)\neq (1,1)}\Big|\int_{\mathbb{R}^3}f^{n+1}v_idv\Big|\Big|\int_{\mathbb{R}^3}f^{n+1}v_jdv\Big|.
\end{align*}
Since $f^n\in\Omega_1$,  we see from Lemma \ref{23} that $R$ can be taken to be arbitrarily small
by taking $\tau$ sufficiently large:
\[
R=O\left(\delta_3,(\ln\tau+1)\tau^{-1}\right).
\]
For  $I$, we use $a^2-b^2=(a-b)(a+b)$ to compute
\begin{align*}
I&\geq\left\{\int_{\mathbb{R}^3}f^{n+1}(|v_1|+v_1)dv\right\}\left\{\int_{\mathbb{R}^3}f^{n+1}(|v_1|-v_1)dv\right\}\cr
&=4\left\{\int_{v_1>0}f^{n+1}|v_1|dv\right\}\left\{\int_{v_1<0}f^{n+1}|v_1|dv\right\}\cr
&\geq4\delta_1^2\left(\int_{v_1>0}e^{-\frac{1}{|v_1|}}f_L|v_1|dv\right)
\left(\int_{v_1<0}e^{-\frac{1}{|v_1|}}f_R|v_1|dv\right)\cr
&=4\delta^2_1\gamma_{\ell,1}.
\end{align*}
In the last line, we used  (\ref{f+}) as
\begin{align*}
f^{n+1}&\geq \delta_1e^{-\frac{x}{\tau|v_1|}}f_L1_{v_1>0}+\delta_1e^{-\frac{1-x}{\tau|v_1|}}f_R1_{v_1<0}\cr
&\geq \delta_1e^{-\frac{1}{|v_1|}}f_L1_{v_1>0}+\delta_1e^{-\frac{1}{|v_1|}}f_R1_{v_1<0}
\end{align*}
and $\tau>1$.
Therefore, for sufficiently large $\tau$, we  get  
\begin{align}\label{T2}
T^{n+1}\geq \frac{1}{3\{\rho^{n+1}\}^2}\left\{4\delta_1^2\gamma_{\ell,1}-C
\left(\frac{\ln \tau+1}{\tau}\right)\right\}\geq \delta^2_1\frac{\gamma_{\ell,1}}{3C_{LR,1}^2}.
\end{align}
Thanks to Lemma \ref{au}.
Finally, we put (\ref{T1}) and (\ref{T2}) into (\ref{equiv T}) to get the desired result.\newline\newline
%
%
\noindent(2) In this critical case, the l.h.s of the equivalence type estimate (\ref{equiv T}) become trivial, and does not give any meaningful information about the positivity of the temperature tensor. Therefore, we need to take a more careful look in the structure of the temperature tensor directly. For this, we observe that the quadratic polynomial of the temperature tensor can be written in terms of the local  energy and the directional local energy in the critical case $\nu=-1/2$:
\begin{align}\label{tbt}
\begin{split}
\left[\kappa^{\top}\left\{\mathcal{T}^{n+1}_{-1/2}\right\}\kappa\right]
&=\frac{1}{\rho^{n+1}}\int_{\mathbb{R}^3}f^{n+1}\left\{|v|^2-(v\cdot\kappa)^2\right\}dv\cr
&-\frac{1}{\rho^{n+1}}\left\{\rho^{n+1}
|U^{n+1}|^2-\rho^{n+1} (U^{n+1}\cdot \kappa)^2\right\}\cr
&\equiv I+II,
\end{split}
\end{align}
for $|\kappa|=1$\newline

\noindent$(i)$ Upper bound of $I+II$: Since, 
$$
|v|^2-(v\cdot\kappa)^2\leq |v|^2\quad\mbox{ and }\quad \rho^{n+1}|U^{n+1}|^2-\rho^{n+1} (U^{n+1}\cdot \kappa)^2 \geq 0,
$$
We can ignore $II$ and employ Lemma \ref{al} and Lemma \ref{au} for $I$ to get 
\begin{align*}
\kappa^{\top}\left\{\mathcal{T}^{n+1}_{-1/2}\right\}\kappa\leq \frac{1}{\rho^{n+1}}\int_{\mathbb{R}^3}f^{n+1}|v|^2dv\leq\frac{1}{a_{\ell,1}\delta_1}
\|f_{LR}\|_{L^1_{\gamma,\langle v\rangle}}\|M_w\|_{L^1_{\gamma,\langle v\rangle}}.\cr
\end{align*}
\noindent$(ii)$ Lower bound of $I+II$: For this, we derive the lower bound for the quadratic polynomial of $\mathcal{T}_{-1/2}$ by combining the lower bound of $I$ and smallness of $II$:\newline

\noindent$(ii$-$a)$ Lower bound of $I$: Since $|v|^2-(v\cdot \kappa)^2\geq0$, we observe from (\ref{f+}), (\ref{f-}) and
Lemma \ref{au} that
\begin{align*}
\rho^{n+1}I
&= \int_{\mathbb{R}^3}f^{n+1}\left\{|v|^2-(v\cdot\kappa)^2dv\right\}\cr
&\geq \delta_1\int_{\mathbb{R}^3}\left\{e^{-\frac{x}{\tau|v_1|}}f_L1_{v_1>0}
+e^{-\frac{1-x}{\tau|v_1|}}f_R1_{v_1<0}\right\}\left\{|v|^2-(v\cdot\kappa)^2\right\}dv\cr
&\geq\delta_1\inf_{|\kappa|=1}\int_{\mathbb{R}^3}e^{-\frac{1}{\tau|v_1|}}f_{LR}\left\{|v|^2-(v\cdot\kappa)^2\right\}dv.\cr
\end{align*}
Since we are assuming $\tau$ is sufficiently large, we assume without loss of generality that $\tau>1$, so that

\begin{align*}
\rho^{n+1}I&\geq\delta_1\inf_{|\kappa|=1}\int_{\mathbb{R}^3}e^{-\frac{1}{|v_1|}}f_{LR}\left\{|v|^2-(v\cdot\kappa)^2\right\}dv= \delta_1a_{-1/2}.\cr
\end{align*}

\noindent$(ii$-$b)$ Smallness of $II$: Thanks to Lemma \ref{23}, we see that $II$ can be 
controlled, up to small error, by the discrepancy of the boundary flux:
\begin{align*}
II&\leq \frac{|\rho^{n+1} U^{n+1}|^2}{\rho^{n+1}}
\leq\frac{1}{a_{\ell,1}}\left|\int_{\mathbb{R}^3}f^{n+1}vdv\right|^2
\leq \frac{1}{a_{\ell,1}}\sum_{i=1}^3\left|\int_{\mathbb{R}^3}f^{n+1}v_idv\right|^2\cr
&\leq2\left|\int_{v_1>0}f_L|v_1|dv-\int_{v_1<0}f_R|v_1|dv\right|^2\cr
&+16\left(\delta_2+\delta_3+\frac{1}{\tau}\right)^2C^2_{LM}+C_{\ell,u}\left(\frac{\ln \tau+1}{\tau}\right)^2.
\end{align*}

From the estimates in $(ii$-$a)$, $(ii$-$b)$,  we have
\begin{align*}
	\begin{split}
		\kappa^{\top}\left\{\mathcal{T}^{n+1}_{-1/2}\right\}\kappa
		&\geq \frac{1}{\rho^{n+1}}\{I-|II|\}\cr 
		&\geq\frac{1}{a_{\ell,1}}
		\left\{a_{-1/2,1}-2\left|\int_{v_1>0}f_L|v_1|dv-\int_{v_1<0}f_R|v_1|dv\right|^2
		-O\left(\delta_2,\delta_3,\tau^{-1}\right)\right\}.
	\end{split}
\end{align*}
Therefore, if  $\delta_2$, $\delta_3$, $\tau^{-1}$ and the flux discrepency:
\[
\left|\int_{v_1>0}f_L|v_1|dv-\int_{v_1<0}f_R|v_1|dv\right|
\]  
are sufficiently small, we get the desired lower bound.
\end{proof}
\section{Cauchy estimate for $f^n$}

The goal of this section is to show that $\{f^n\}$ forms a Cauchy sequence in an appropriate norm. First, we consider the continuity property of the ellipsoidal Gaussian:
\begin{proposition}\label{Lipshitz} Let $f$, $g$ be elements of $\Omega_1$ $(-1/2<\nu<1)$ or  $\Omega_2$ $(\nu=-1/2)$. Then the non-isotropic Gaussian $\mathcal{M}_{\nu}$ satisfies 
\begin{eqnarray*}
|\mathcal{M}_{\nu}(f)-\mathcal{M}_{\nu}(g)|\leq C\sup_x\|f-g\|_{L^{1}_2}
e^{-C|v|^2}.
\end{eqnarray*}
\end{proposition}
\begin{proof}
	The case of $-1/2<\nu<1$ is covered in \cite{Bang Y}. Here we
        only consider the case $\nu=-1/2$. Throughout this proof,
        $\nu$ is fixed to be $-1/2$. 
	 We apply Taylor expansion to $\mathcal{M}_{\nu}(f)-\mathcal{M}_{\nu}(g)$ as
	\begin{align}\label{turnbackto}
	\begin{split}
	&\mathcal{M}_{\nu}(f)-\mathcal{M}_{\nu}(g)\cr
	&\quad=(\rho_{f}-\rho_{g})
	\int^1_0\frac{\partial\mathcal{M}_{\nu}(\theta)}{\partial \rho}d\theta
	+(U_{f}-U_{g})\int^1_0\frac{\partial\mathcal{M}_{\nu}(\theta)}{\partial U}d\theta
	+(\mathcal{T}_{f}-\mathcal{T}_{g})\int^1_0\frac{\partial\mathcal{M}_{\nu}(\theta)}{\partial \mathcal{T}_{\nu}}d\theta\cr
	&\quad\equiv (\rho_{f}-\rho_{g})I_1+(U_{f}-U_{g})I_2+(\mathcal{T}_{f}-\mathcal{T}_{g})I_3,
	\end{split}
	\end{align}
	where we used abbreviated notation:
	\begin{eqnarray*}
		\frac{\partial\mathcal{M}_{\nu}(\theta)}{\partial X}=\frac{\partial\mathcal{M}_{\nu}}{\partial X}
		(\rho_{\theta}, U_{\theta}, \mathcal{T}_{\theta})
	\end{eqnarray*}
	and $(\rho_{\theta}, U_{\theta}, \mathcal{T}_{\theta})=(1-\theta)\big(\rho_{f}, U_{f}, \mathcal{T}_{f}\big)
	+\theta\big(\rho_{g}, U_{g}, \mathcal{T}_{g}\big)$.
	Since the transitional macroscopic fields $(\rho_{\theta}, U_{\theta}, \mathcal{T}_{\theta})$ are all linear combinations of macroscopic fields of $f$ and $g$, all the estimates for the macroscopic fiels given in $(\mathcal{A}_2)$, $(\mathcal{B}_2)$, $(\mathcal{C}_2)$ and $(\mathcal{D}_2)$ ($i=1,2$) hold the same for the transitional macroscopic fields too. Therefore, we will refer to the corresponding properties of $\Omega_i$ $(i=1,2)$ for $\rho$, $U$, $\mathcal{T}_{\nu}$, whenever such estimates are needed for $(\rho_{\theta}, U_{\theta}, \mathcal{T}_{\theta})$.

	%
	%
	%
	%
	\noindent(a) Estimate for $I_1$: Since
	\begin{eqnarray*}
		\frac{\partial\mathcal{M}_{\nu}(\theta)}{\partial\rho}=\frac{1}{\rho_{\theta}}\mathcal{M}_{\nu}(\theta),
	\end{eqnarray*}
	we see from $(\mathcal{B}_2)$ of $\Omega_2$ and Lemma \ref{Max decomposition} that
	\begin{eqnarray}
	I_1\leq  \frac{1}{\delta_1}C_{\ell,u}e^{-C_{\ell,u}|v|^2}.
	\end{eqnarray}
	\noindent(b) Estimate for $I_2$: An explicit computation gives
	\begin{eqnarray*}
		\frac{\partial\mathcal{M}_{\nu}(\theta)}{\partial U}
		=-\frac{1}{2}\Big\{(v-U_{\theta})^{\top}\mathcal{T}^{-1}_{\theta}+\mathcal{T}^{-1}_{\theta}(v-U_{\theta})\Big\}
		\mathcal{M}_{\nu}(\theta).
	\end{eqnarray*}
	Put $X=v-U_{\theta}$ and recall the property ($\mathcal{C}_2$) of $\Omega_2$ to compute
	\begin{align}\label{similar}
		\begin{split}
		|X^{\top}\mathcal{T}_{\theta}^{-1}|&=\sup_{|Y|=1}X^{\top}\{\mathcal{T}_{\theta}\}^{-1}Y\cr
		&=\frac{1}{2}\sup_{|Y|=1}\Big\{(X+Y)^{\top}\{\mathcal{T}_{\theta}\}^{-1}(X+Y)-X^{\top}\{\mathcal{T}_{\theta}\}^{-1}X-Y^{\top}\{\mathcal{T}_{\theta}\}^{-1}Y\Big\}\cr
		&\leq \frac{2a_{\ell,1}}{\delta_1a_{-1/2,1}}\sup_{|Y|=1}\left(|X+Y|^2+|X|^2+|Y|^2\right)\cr
		&\leq \frac{8a_{\ell,1}}{\delta_1a_{-1/2,1}}(1+|X|^2),
		\end{split}
	\end{align}
	so that
	\begin{align}\label{so that}
		|\{\mathcal{T}_{\theta}\}^{-1}(v-U_{\theta })|\leq C\frac{8a_{\ell,1}}{\delta_1a_{-1/2,1}}(1+|v|^2).
	\end{align}
	Therefore, in view of Lemma \ref{Max decomposition}, we have
	\begin{eqnarray*}
		\Big|\frac{\partial\mathcal{M}_{\nu}(\theta)}{\partial U}\Big|
		\leq  C_{\ell,u}e^{-C_{\ell,u}|v|^2}.
	\end{eqnarray*}
	\noindent(c) Estimate for $I_3$: We compute
	\begin{align}\label{plug}
		\frac{\partial \mathcal{M}_{\nu}(\theta)}{\partial \mathcal{T}_{ij}}
		=\frac{1}{2}\left[-\frac{1}{\det\mathcal{T}_{\theta}}\frac{\partial\det\mathcal{T}_{\theta}}{\partial\mathcal{T}_{\theta ij}}
		+(v-U_{\theta})^{\top}\mathcal{T}_{\theta}^{-1}\left(\frac{\partial\mathcal{T}_{\theta}}{\partial{\mathcal{T}_{ij}}}\right)\mathcal{T}_{\theta}^{-1}(v-U_{\theta})\right]
		\mathcal{M}_{\nu}(\theta)
	\end{align}
	 and observe that, for each pair of $(i,j)$, $\frac{\partial\mathcal{T}_{\theta}}{\partial{\mathcal{T}_{\theta ij}}}$ is either 1 or 0, so that
	\begin{eqnarray}\label{use3}\begin{split}
	\qquad\left|(v-U_{\theta})^{\top}\mathcal{T}^{-1}_{\theta}\left(\frac{\partial\mathcal{T}_{\theta}}{\partial{\mathcal{T}_{\theta ij}}}\right)\mathcal{T}^{-1}_{\theta}
	(v-U_{\theta})\right|
	&\leq \left|(v-U_{\theta})^{\top}\mathcal{T}^{-1}_{\theta}\right|
	\left|\mathcal{T}^{-1}_{\theta}(v-U_{\theta})\right|\cr
	\qquad&\leq C_{\ell,u}(1+|v|^2).
	\end{split}
	\end{eqnarray}
Here, we used (\ref{so that}).	
	On the other hand, to estimate the derivatives of the determinant, we write $\frac{\partial\det\mathcal{T}_{\theta}}{\partial\mathcal{T}_{\theta ij}}$ as follows:
	\[
	\sum_{i,j,m,n} C_{ijmn}\mathcal{T}_{\theta ij}\mathcal{T}_{\theta mn}
	\]
	for some constants $C_{ijmn}$. We then note that Lemma \ref{gl} (2) implies
	\[
	\mathcal{T}_{\theta ij}\leq \frac{8}{2a_{\ell,1}\delta_1}C_{LR,1}.
	\]
	Indeed,  as in (\ref{similar}), we see that
		\begin{align*}
		\begin{split}
			|\mathcal{T}_{\theta ij}|&=|e^{\top}_i\mathcal{T}_{\theta}e_j|\cr
			&=\frac{1}{2}\sup_{|Y|=1}\Big|\Big\{(e_i+e_j)^{\top}\mathcal{T}_{\theta}(e_i+e_j)-e_i^{\top}\mathcal{T}_{\theta}e_i-e_j^{\top}\mathcal{T}_{\theta}e_j\Big\}\Big|\cr
			&\leq \frac{2a_{\ell,1}}{\delta_1a_{-1/2,1}}\left(|e_i+e_j|^2+|e_i|^2+|e_j|^2\right)\cr
			&= \frac{8a_{\ell,1}}{\delta_1a_{-1/2,1}}.
		\end{split}
	\end{align*}
Therefore, we can estimate
	\[
		\left|\frac{\partial\det\mathcal{T}_{\theta}}{\partial\mathcal{T}_{\theta ij}}\right|\leq \left(\frac{3}{2a_{\ell,1}\delta_1}C_{LR,1}\right)^2.
	\]
	Finally, we recall ($\mathcal{B}_2$) of $\Omega_2$ to find
\begin{align}\label{det2}
	\det \mathcal{T}^{}_{\theta}\geq	\left(\delta_1\frac{a_{-1/2}}{2C_{LR,1}}\right)^3,\qquad (\nu=-1/2)
\end{align}

We plug (\ref{use3}), (\ref{hence}) and (\ref{det2}) into (\ref{plug}), and employ
Lemma \ref{Max decomposition} to get
\begin{align}\label{hence}
	\Big|\frac{\partial \mathcal{M}_{\nu}(\theta)}{\partial \mathcal{T}_{\theta ij}}\Big|\leq C\delta_1^{-8}(1+|v|^2)\mathcal{M}_{\nu}(\theta)\leq Ce^{-C|v|^2}.
\end{align}
(Note that $\delta_1$ is not small in the inflow dominant case.)
We now turn back to (\ref{turnbackto}) with these estimates to derive
	\begin{align}\label{obtain}
		\begin{split}
			&|\mathcal{M}_{\nu}(f)-\mathcal{M}_{\nu}(g)|
			\leq C\Big\{|\rho_{f}-\rho_{g}|+|U_{f}-U_{g}|
			+|\mathcal{T}_{f}-\mathcal{T}_{g}|\Big\} e^{-C|v|^2}.
		\end{split}
	\end{align}
For the difference of macroscopic fields, we treat as
	\begin{eqnarray*}
		|\rho_{f}-\rho_{g}|=\int_{\mathbb{R}^3} |f-g|dv\leq C\sup_x\|f-g\|_{L^1_2},
	\end{eqnarray*}
	\begin{eqnarray*}
		|U_{f}-U_{g}|
		\leq\frac{1}{\rho_f}|\rho_fU_f-\rho_gU_g|+\frac{1}{\rho_f}|\rho_f-\rho_g||U_g|
		\leq C_{\ell,u}{\delta_1^{-1}}\sup_x\|f-g\|_{L^{1}_2},
	\end{eqnarray*}
	and
	\begin{eqnarray*}
		|\mathcal{T}_{f}-\mathcal{T}_{g}|
		\leq\frac{1}{\rho_f}|\rho_f\mathcal{T}_{f}-\rho_g\mathcal{T}_{g}|
		+\frac{1}{\rho_f}|\rho_f-\rho_g||\mathcal{T}_{g}|=C_{\ell,u}{\delta^{-1}_1}\sup_x\|f-g\|_{L^{1}_2}.
	\end{eqnarray*}
This completes the proof.
\end{proof}

\begin{proposition}\label{fixed point 2}
Suppose $f^{n}, \, f^{n+1}\in\Omega_i$. $(i=1,2)$ Then, under the assumption of Theorem \ref{Main1}, we have 
\begin{align*}
&\sup_{x}\|f^{n+1}-f^{n}\|_{L^1_{2}}+\|f^{n+1}-f^{n}\|_{L^1_{\gamma,|v_1|}}+\|f^{n+1}-f^{n}\|_{L^1_{\gamma,\langle v\rangle}}\cr
&\qquad \preceq \left(\frac{\ln \tau+1}{\tau }\right)\sup_{x}\|f^{n}-f^{n-1}\|_{L^1_2}
+(\delta_2+\delta_3)\|f^{n}-f^{n-1}\|_{L^1_{\gamma,|v_1|}}+\delta_3\|f^{n}-f^{n-1}\|_{L^1_{\gamma,\langle v\rangle}}
\end{align*}

\end{proposition}
\begin{proof}

$(1)$ Estimates in the trace norm $\|\cdot\|_{L^1_{\gamma,|v_1|}}$:
First, we note from our boundary condition that, for $v_1>0$
\begin{align*}
\begin{split}
f^{n+1}(0,v)-f^n(0,v)&=\delta_2\left(\int_{v_1<0}\left\{f^n(0,v)-f^{n-1}(0,v)\right\}|v_1|dv\right)M_{w}(0)\cr
&+\delta_3\left\{f^n(0,Rv)-f^{n-1}(0,Rv)\right\}.
\end{split}
\end{align*}
Taking integration w.r.t $|v_1|dv$,
\begin{align}\label{sum1}
	\begin{split}
		&\int_{v_1>0}|f^{n+1}(0,v)-f^n(0,v)||v_1|dv\cr
		&\qquad\leq\delta_2\left(\int_{v_1<0}|f^n(0,v)-f^{n-1}(0,v)||v_1|dv\right)\int_{v_1>0}M_w(0)|v_1|dv\cr
		&\qquad+\delta_3\int_{v_1>0}|f^n(0,Rv)-f^{n-1}(0,Rv)||v_1|dv\cr
		&\qquad\leq(\delta_2+\delta_3)\int_{v_1<0}|f^n(0,v)-f^{n-1}(0,v)||v_1|dv.	
	\end{split}
\end{align}
On the other hand, for $v_1<0$, we have from (\ref{f-})
\[
f^{n+1}(0,v)=I(f^n)+II(f^n),
\]
where 
\begin{eqnarray*}
I(f)=e^{-\frac{x}{\tau|v_1|}}f(1,v),\quad
II(f)=\frac{1}{\tau|v_1|}\int^1_0e^{-\frac{x-y}{\tau|v_1|}}\mathcal{M}_{\nu}(f)dy,
\end{eqnarray*}
where for the sake of clarity $\rho_f$ is the local density associated
to $f$ and $\rho_g$ is the local density associated to $g$.

\noindent $(1$ - $i)$ The estimate for $I(f^n)-I(f^{n-1})$: Since
\begin{align*}
I(f^{n})-I(f^{n-1})
&=\delta_2 e^{-\frac{x}{\tau|v_1|}}M_w(1)\left(\int_{v_1>0}\left\{f^n(1,v)-f^{n-1}(1,v)\right\}|v_1|dv\right)\cr
&+\delta_3 e^{-\frac{x}{\tau|v_1|}}\left\{f^n(1,Rv)-f^{n-1}(1,Rv)\right\},
\end{align*}

we have
\begin{align}\label{sum2}
	\begin{split}
&\int_{v_1<0}|I(f^{n})-I(f^{n-1})|v_1|dv\cr
&\quad\geq \delta_2\left\{\int_{v_1<0}e^{-\frac{a_{\ell,1}}{\tau|v_1|}}M_w(v)|v_1|dv\right\}\left\{\int_{v_1>0}|f^n(1,v)-f^{n-1}(1,v)||v_1|dv\right\}\cr
&\quad+\delta_3\left\{\int_{v_1>0}|f^n(1,v)-f^{n-1}(1,v)||v_1|dv\right\}\cr
&\quad\leq(\delta_2+\delta_3) 
\left\{\int_{v_1>0}|f^n(1,v)-f^{n-1}(1,v)||v_1|dv\right\}
\end{split}
\end{align}

$\bullet$ The estimate for $II(f)-II(g)$: 


We recall Lemma \ref{main estimate} and Proposition \ref{Lipshitz} to estimate 
\begin{align}\label{sum5}
	\begin{split}
&\hspace{-0.5cm}\int_{v_1>0}|II(f^n)-II(f^{n-1})|v_1|dv\cr
& \leq \int_{v_1>0}\frac{1}{\tau|v_1|}\int^x_0e^{-\frac{1}{\tau|v_1|}}| \mathcal{M}_{\nu}(f^n)-\mathcal{M}_{\nu}(f^{n-1})||v_1|dydv\cr
&\leq C\left\{\int^1_0\int_{v_1>0}\frac{1}{\tau|v_1|}e^{-\frac{x-y}{\tau|v_1|}}
e^{-C_{\ell,u}|v|^2}dv_1dy\right\}
\sup_{x}\|f^n-f^{n-1}\|_{L^1_2}\cr
&\leq C\left(\frac{\ln \tau+1}{\tau}\right)\sup_{x}\|f^n-f^{n-1}\|_{L^1_2}.
\end{split}
\end{align}
We sum up (\ref{sum1}), (\ref{sum2}) and (\ref{sum5}) to obtain
\begin{align*}
\|f^{n+1}-f^{n}\|_{L^1_{\gamma,|v_1|,-}}
&\preceq\left(\frac{\ln \tau+1}{\tau }\right)\sup_{x}\|f^{n}-f^{n-1}\|_{L^1_{2}}+(\delta_2+\delta_3)\|f^{n}-f^{n-1}\|_{L^1_{\gamma,|v_1|,+}}.
\end{align*}
In an almost identical manner, we can derive
\begin{align*}
	\|f^{n+1}-f^{n}\|_{L^1_{\gamma,|v_1|,+}}
	&\preceq \left(\frac{\ln \tau+1}{\tau }\right)\sup_{x}\|f^{n}-f^{n-1}\|_{L^1_{2}}
	+(\delta_2+\delta_3)\|f^{n}-f^{n-1}\|_{L^1_{\gamma,|v_1|,-}}.
\end{align*}
Therefore, 
\begin{align}\label{final1}
	\begin{split}
	\|f^{n+1}-f^{n}\|_{L^1_{\gamma,|v_1|}}
	& \preceq\left(\frac{\ln \tau+1}{\tau }\right)\sup_{x}\|f^{n}-f^{n-1}\|_{L^1_{2}}
	+(\delta_2+\delta_3)\|f^{n}-f^{n-1}\|_{L^1_{\gamma,|v_1|}}.\cr
\end{split}
\end{align}
(2) Estimates in the trace norm $\|\cdot\|_{L^1_{\gamma,\langle v_1\rangle}}$: We only estimate the boundary term $I$, since the estimates for $II$ are almost identical.
By an almost identical calculation, we arrive at 
\begin{align*}
	\begin{split}
		&\int_{v_1<0}|I(f^{n})-I(f^{n-1})|\langle v\rangle dv\cr
		&\quad \leq \delta_2\left\{\int_{v_1<0}e^{-\frac{a_{\ell,1}}{\tau|v_1|}}M_w(1)\langle v\rangle dv\right\}\left\{\int_{v_1>0}|f^n(1,v)-f^{n-1}(1,v)||v_1|dv\right\}\cr
		&\quad+\delta_3\left\{\int_{v_1>0}|f^n(1,v)-f^{n-1}(1,v)||v_1|dv\right\}\cr
		&\quad\preceq(\delta_2+\delta_3) 
		\int_{v_1>0}|f^n(1,v)-f^{n-1}(1,v)|\langle v\rangle dv.
	\end{split}
\end{align*}
Then, through similar computaitons for $II$ terms, (we omit the proof to avoid repetitions.) we can obtain the estimates in $\|\cdot\|_{L^1_{\gamma,\langle v\rangle}}$:
\begin{align}\label{final2}
	\begin{split}
	\|f^{n+1}-f^{n}\|_{L^1_{\gamma,\langle v\rangle}}
	&\preceq\left(\frac{\ln t+1}{\tau}\right)\sup_{x}\|f^{n}-f^{n-1}\|_{L^1_{2}}\cr
	&+\delta_2\|f^{n}-f^{n-1}\|_{L^1_{\gamma,|v_1|}}
	+\delta_3\|f^{n}-f^{n-1}\|_{L^1_{\gamma,\langle v\rangle}} .
\end{split}
\end{align}
The estimates in $\sup_x\|\cdot\|_{L^1_2}$ can be derived similarly:
\begin{align}\label{final3}
	\begin{split}
		\|f^{n+1}-f^{n}\|_{L^1_2}
			&\preceq
                        \left(\frac{\ln t+1}{\tau}\right)\sup_{x}\|f^{n}-
f^{n-1}\|_{L^1_{2}}\cr
		&+\delta_2\|f^{n}-f^{n-1}\|_{L^1_{\gamma,|v_1|}}
		+\delta_3\|f^{n}-f^{n-1}\|_{L^1_{\gamma,\langle v\rangle}}.
	\end{split}
\end{align}
The estimates (\ref{final1}), (\ref{final2}) and (\ref{final3}) give the desired result.
\end{proof}

%
%
%
%
%


\section{Proof of Theorem \ref{Main2}: The diffusive boundary condition}
We now turn to the proof of Theorem \ref{Main2}. Since many parts overlap with the proof of Theorem \ref{Main1}, we focus on the difference of the proof. We start with the reformulation of the problem.
\subsection{Reformulation of the problem} 
Consider the following mild formulation of (\ref{ESBGK}):
\begin{align*}
	\begin{split}
		f(x,v)&=\delta_1 f_L+\delta_2\left(\int_{v_1<0}f(0,v)|v_1|dv\right)M_w(0)+\delta_3f(0,Rv)\cr
		&+\frac{1}{\tau|v_1|}\int^x_0
		\mathcal{R}(y,v)dy,\quad (v_1>0)\cr 
		f(x,v)&=\delta_1 f_R+\delta_2\left(\int_{v_1>0}f(1,v)|v_1|dv\right)M_w(1)+\delta_3f(1,Rv)\cr
		&+\frac{1}{\tau|v_1|}\int^1_x\mathcal{R}(y,v)dy,\quad (v_1<0)
	\end{split}
\end{align*}
so that
\begin{align*}
	\begin{split}
		f(1,v)&=\delta_1 f_L+\delta_2\left(\int_{v_1<0}f(0,v)|v_1|dv\right)M_w(0)+\delta_3f(0,Rv)\cr
		&+\frac{1}{\tau|v_1|}\int^1_0
		\mathcal{R}(y,v)dy,\quad (v_1>0)\cr
		f(0,v)&=\delta_1 f_R+\delta_2\left(\int_{v_1>0}f(1,v)|v_1|dv\right)M_w(1)+\delta_3f(1,Rv)\cr
	&+\frac{1}{\tau|v_1|}\int^1_0\mathcal{R}(y,v)dy,\quad (v_1<0).
	\end{split}
\end{align*}
Integrating with respect to $|v_1|dv$:
\begin{align}\label{this}
	\begin{split}
		\int_{v_1>0}f(1,v)|v_1|dv&
		=\delta_1 \int_{v_1>0}f_L|v_1|dv + (\delta_2+\delta_3)\int_{v_1<0}f(0,v)|v_1|dv\cr&
		+\frac{1}{\tau}\int_{v_1>0}\int^1_0\mathcal{R}(y,v)dydv,\cr
		\int_{v_1<0}f(0,v)|v_1|dv&=\delta_1 \int_{v_1<0}f_R|v_1|dv+(\delta_2+\delta_3)\int_{v_1>0}f(1,v)|v_1|dv\cr&
		+\frac{1}{\tau}\int_{v_1<0}\int^1_0\mathcal{R}(y,v)dydv,
	\end{split}
\end{align}
where
\begin{align*}
\mathcal{R}(f)(x,v)=\mathcal{M}_{\nu}(f)(x,v)-f(x,v),
\end{align*}
throughout this section.
Inserting $(\ref{flux c})$ into $(\ref{this})_1$, we get
\begin{align*}
\int_{v_1<0}f(0,v)|v_1|dv&=\frac{1-\delta_1}{2-\delta_1}+\frac{\delta_1}{2-\delta_1}\int_{v_1<0}f_R|v_1|dv
-\frac{1}{\tau(2-\delta_1)}\int_{v_1>0}\int^1_0\mathcal{R}(y,v)dydv.
\end{align*}
Similarly, from $(\ref{flux c})$ and $(\ref{this})_2$, we get
\begin{align*}
\int_{v_1>0}f(1,v)|v_1|dv&=\frac{1- \delta_1}{2-\delta_1}+\frac{\delta_1}{2-\delta_1}\int_{v_1>0}f_L|v_1|dv
-\frac{1}{\tau(2-\delta_1)}\int_{v_1<0}\int^1_0\mathcal{R}(y,v)dydv.
\end{align*}
From this, we derive the new formulation of the problem given in Definition 1.3.

%
%
%
%
\subsection{Approximation scheme and solution spaces}
We construct the solution for (\ref{ESBGK}) from the following approximate scheme:
\begin{align}\label{f+ dff}
	\begin{split}
		f^{n+1}(x,v)&=e^{-\frac{x}{\tau|v_1|}}f^n(0,v)
		+\frac{1}{\tau|v_1|}\int_{0}^{x} e^{-\frac{x-y}{\tau|v_1|}}\mathcal{M}_{\nu}(f^n)dy,
		\quad\text{if $v_{1}>0$}
	\end{split}
\end{align}
and
\begin{align}\label{f- dff}
	\begin{split}
		f^{n+1}(x,v)&=e^{-\frac{x}{\tau|v_1|}}f^n(1,v)
		+\frac{1}{\tau|v_1|}\int_{x}^1 e^{-\frac{x-y}{\tau|v_1|}}
		\mathcal{M}_{\nu}(f^n)dy,
		\quad\text{if $v_{1}<0$}
	\end{split}
\end{align}
where
\begin{align}\label{f_delta dff}
	\begin{split}
		f^{n+1}(0,v)&=\delta_1f_L(v)+\delta_2\mathcal{S}_L(f^n)M_w(0)+\delta_3 f^n(0,Rv),\quad (v_1>0)\cr
		f^{n+1}(1,v)&=\delta_1f_R(v)+\delta_2\mathcal{S}_R(f^n)M_w(1)+\delta_3 f^n(1,Rv),\quad (v_1<0)
	\end{split}
\end{align}
and 
\begin{align*}
	\begin{split}
	\mathcal{S}_L(f^n)&=\frac{1-\delta_1}{2-\delta_1}+\frac{\delta_1}{2-\delta_1}\int_{v_1<0}f_R|v_1|dv
	-\frac{1}{\tau(2-\delta_1)}\int_{v_1>0}\int^1_0\mathcal{R}^n(y,v)dydv\cr
	\mathcal{S}_R(f^n)&=\frac{1-\delta_1}{2-\delta_1}+\frac{\delta_1}{2-\delta_1}\int_{v_1>0}f_L|v_1|dv
	-\frac{1}{\tau(2-\delta_1)}\int_{v_1<0}\int^1_0\mathcal{R}^n(y,v)dydv
\end{split}
\end{align*}
with
\begin{align*}
		\mathcal{R}^n(y,v)=\rho^n\left\{\mm(f^n)-f^n\right\}.
\end{align*}
As in the inflow dominant case, we define two function spaces. First we define the function space for the non-critical case $-1/2<\nu<1$:
\begin{align*}
	\Omega_3=\Big\{f\in L^{\infty}\left([0,1]; L^1_2(\mathbb{R}^3)\right)\cap L^1_{\gamma,\langle v\rangle}(\mathbb{R}^3)~|&~f  \mbox{ satisfies } (\mathcal{A}_3), (\mathcal{B}_3), (\mathcal{C}_3), (\mathcal{D}_3)
	\Big\}
\end{align*}
 where $(\mathcal{A}_3)$, $(\mathcal{B}_3)$, $(\mathcal{C}_3)$ and $(\mathcal{D}_3)$ denote
\begin{itemize}
	\item ($\mathcal{A}_3$) $f$ is non-negative:
	\[
	f(x,v)\geq0 \mbox{ for }x,v\in [0,1]\times \mathbb{R}^3.
	\]
	\item ($\mathcal{B}_3$) The macroscopic field is well-defined:
	\begin{align*}
		&\int_{\mathbb{R}^3}f(x,v)dv\geq a_{\ell,2},\quad \int_{\mathbb{R}^3}f(x,v)(1+|v|^2)dv\leq 
		2C_{LR,2}.
	\end{align*}
	\item ($\mathcal{C}_3$) The temperature tensor is well-defined:
	\begin{align*}
		C^1_{\nu}\delta_1^2\frac{\gamma_{\ell,2}}{3C_{LR,2}^2}\leq\kappa^{\top}\left\{\mathcal{T}_{\nu}\right\}\kappa\leq \frac{2 }{{3a_{\ell,2}}}C^2_{\nu}C_{LR,2}.
	\end{align*}
	\item ($\mathcal{D}_3$) The inflow data satisfies:
	\begin{align*}
		&\|f\|_{L^1_{\gamma,|v_1|,\pm}}\leq 2\big(1+\|f_{LR}\|_{L^1_{\gamma,|v_1|}}\big),\quad\|f\|_{L^1_{\gamma,\langle v\rangle}}\leq 2C_{LR,2}.
	\end{align*}
\end{itemize}
For the critical case $\nu=-1/2$, we define 
\begin{align*}
	\Omega_4=\Big\{f\in L^{\infty}\left([0,1]; L^1_2(\mathbb{R}^3)\right)\cap L^1_{\gamma,\langle v\rangle}(\mathbb{R}^3)~|&~f  \mbox{ satisfies } (\mathcal{A}_4), (\mathcal{B}_4), (\mathcal{C}_4), (\mathcal{D}_4)
	\Big\}
\end{align*}
 where $(\mathcal{A}_4)$, $(\mathcal{B}_4)$, $(\mathcal{C}_4)$ and $(\mathcal{D}_4)$ denote
\begin{itemize}
	\item ($\mathcal{A}_4$) $f$ is non-negative:
	\[
	f(x,v)\geq0 \mbox{ for }x,v\in [0,1]\times \mathbb{R}^3.
	\]
	\item ($\mathcal{B}_4$) The macroscopic field is well-defined:
	\begin{align*}
		&\int_{\mathbb{R}^3}f(x,v)dv\geq a_{\ell,2},\quad \int_{\mathbb{R}^3}f(x,v)(1+|v|^2)dv\leq 2C_{LR,2}.
	\end{align*}
	\item ($\mathcal{C}_4$) The temperature tensor is well-defined:
	\begin{align*}
		\delta_2\frac{a_{-1/2}}{2C_{LR,2}}\leq\kappa^{\top}\left\{\mathcal{T}_{-1/2}\right\}\kappa\leq \frac{3}{2a_{\ell,2}}C_{LR,2}.
	\end{align*}
	\item ($\mathcal{D}_4$) The inflow data satisfies:
	\begin{align*}
		\|f\|_{L^1_{\gamma,|v_1|,\pm}}&\leq 2\big(1+\|f_{LR}\|_{L^1_{\gamma,|v_1|}}\big),\quad\|f\|_{L^1_{\gamma,\langle v\rangle,\pm}}\leq 2C_{LR,2}.
	\end{align*}
\end{itemize}

%
%
%
%

Before we move on to the proof of uniform estimates for $f^n$, we recored a few estimates that will be fruitfully used throughout the paper. The proof for the Lemma \ref{Max decomposition dff}  is almost identical to the corresponding estimates in
Lemma \ref{Max decomposition}, and we omit the proof.
\begin{lemma}\label{Max decomposition dff} $(1)$ Let $f\in\Omega_3$. Then there exist  positive constants $C$  depending only on the quantities (\ref{quantities}) and $\gamma_{\ell,2}$ such that
	\[
	\mathcal{M}_{\nu}(f)\leq  Ce^{-C|v|^2}.
	\]
	$(2)$ Let $f\in\Omega_4$. Then there exists  positive constants $C$  depending only on the quantities (\ref{quantities}) and $a_{-1/2,2}$ such that
	\[
	\mathcal{M}_{\nu}(f)\leq  Ce^{-C|v|^2}.
	\]
	
\end{lemma}

%
%
%
%

%
%
%
%
\subsection{ $f^n\in\Omega_{i}$ $(i=3,4)$ for all $n$ }
The main result of this section is the following proposition
\begin{proposition}\label{fixed point 1 diff}
	\noindent$(1)$ Let $-1/2<\nu<1$. Assume $f_{LR}$ satisfies the conditions of Theorem \ref{Main2} (1). Then $f^{n}\in \Omega_3$ for all $n$.\newline
	\noindent$(2)$ Let $\nu=-1/2$. Assume $f_{LR}$ satisfies the conditions of Theorem \ref{Main2} (2). Then, $f^{n}\in \Omega_4$ for all $n$.
\end{proposition}
We divide the proof into Lemma \ref{geq dff}, \ref{al dff}, \ref{au dff}, and Lemma \ref{gl dff}.
\begin{lemma}\label{geq dff} Let $f^n\in \Omega_3$ or $\Omega_4$. Then, for sufficiently small $\delta_1$ and sufficiently large $\tau$, we have
	\[
	f^{n+1}\geq0.
	\]
\end{lemma}

\begin{proof}
	Since $f^n\in\Omega_{3}$, we have from Lemma \ref{main estimate}
\begin{align*}
&\mathcal{S}_L(f^n)=\frac{1-\delta_1}{2-\delta_1}+\frac{\delta_1}{2-\delta_1}\int_{v_1>0}f_L|v_1|dv
	-\frac{1}{\tau(2-\delta_1)}\int_{v_1>0}\int^1_0\mathcal{R}^n(y,v)dydv\cr
&\qquad\geq \frac{1-\delta_1}{2-\delta_1}+\delta_1\int_{v_1>0}f_L|v_1|dv
-\frac{1}{\tau}\int_{\rr}\int^1_0\left(\mm(f^n)+f^n\right)(1+|v|^2)dydv\cr
&\qquad\geq \frac{1}{3}
\end{align*}
for sufficiently small $\delta_1$ and $\tau^{-1}$.
Similarly, $\mathcal{S}_R(f^n)\geq 1/3$.
Therefore,
\begin{align*}
f^{n+1}&\geq \frac{1}{3}\delta_2 e^{-\frac{x}{\tau|v_1|}} M_w(0)1_{v_1>0}
+\frac{1}{3}\delta_2
         e^{-\frac{1-x}{\tau|v_1|}} M_w(1)1_{v_1<0}\cr
&\geq \frac{1}{3}\delta_2e^{-\frac{1}{|v_1|}}M_w\geq 0
\end{align*}
	for $v_1>0$.
	The case for $v_1<0$ is the same.
\end{proof}

\begin{lemma} \label{al dff} Assume $f\in \Omega_3$ or $\Omega_4$. Then we have
	\begin{align*}
		\int_{\mathbb{R}^3}f^{n+1}dv\geq \delta_2 a_{\ell,2}.
	\end{align*}
\end{lemma}

\begin{proof}
	We only prove the second one. Recall from the previous proof that
	\begin{align*}
		f^{n+1}\geq
          \frac{1}{3}\delta_2e^{-\frac{1}{|v_1|}}   M_w.
	\end{align*}
	Integrating with respect to $v$, we obtain the desired lower bound.
\end{proof}
\begin{lemma}\label{auu}
	\noindent$(1)$ Let $f^n\in \Omega_1$ or $\Omega_2$. Then we have
	\begin{align*}
	\|f^{n+1}\|_{L^1_{\gamma,|v_1|,+}},~ \|f^{n+1}\|_{L^1_{\gamma,|v_1|,-}}\leq 2\big(1+\|f_{LR}\|_{L^1_{\gamma,|v_1|}}\big).
	\end{align*}
	$(2)$ Let $f^n\in \Omega_1$ or $\Omega_2$. Then we have
	\begin{align*}
	\|f^{n+1}\|_{L^1_{\gamma,\langle v\rangle,+ }},~ \|f^{n+1}\|_{L^1_{\gamma,\langle v\rangle,- }}\leq 2\big(\|f_{LR}\|_{L^1_{\gamma,\langle v\rangle }}+\|M_w\|_{L^1_{\gamma,\langle v\rangle }}\big).
	\end{align*}
\end{lemma}
\begin{proof} (1) $\bullet$ Estimate for outflux $\|f^{n+1}\|_{L^1_{\gamma,|v_1|,+}}$: 
	Using (\ref{f+ dff}), we can write $f^{n+1}(0,v)$ for $v_1<0$ as
	\begin{align}\label{int22}
	\begin{split}
	f^{n+1}(0,v)&=\delta_1 f_R+\delta_2\mathcal{S}_RM_w(1)+\delta_3 f^n(1,Rv)
	+\frac{1}{\tau|v_1|}\int_{0}^1 e^{-\frac{y}{\tau|v_1|}}
	\mathcal{M}_{\nu}(f^n)dy
	\end{split}
	\end{align}
	which, in view of  Lemma \ref{main estimate}, yields
	\begin{align}\label{f222}
	\begin{split}
	&\int_{v_1<0}f^{n+1}(0,v)|v_1|dv\cr
	&\leq \delta_1\int_{v_1<0}f_R|v_1|dv+\delta_2\int_{v_1>0}M_w(1)|v_1|dv+\delta_3\int_{v_1>0}f^{n}(1,v)|v_1|dv\cr
	&+C_{\ell,u}\left(\frac{\ln \tau+1}{\tau}\right).
	\end{split}
	\end{align}
	Here, we used $\mathcal{S}_L\leq 1$, which follows directly from the smallness of $\delta_1$ and Lemma \ref{main estimate}.
	Similarly,
	\begin{align}\label{f232}
	\begin{split}
	&\int_{v_1>0}f^{n+1}(1,v)|v_1|dv\cr
	&\leq \delta_1\int_{v_1>0}f_L|v_1|dv+\delta_2\int_{v_1<0}M_w(0)|v_1|dv+\delta_3\int_{v_1<0}f^{n}(0,v)|v_1|dv\cr
	&+C_{\ell,u}\left(\frac{\ln \tau+1}{\tau}\right).
	\end{split}
	\end{align}
	From (\ref{f222}) and $(\ref{f232})$, we obtain
	\begin{align}\label{ff112}
	\begin{split}
	\|f^{n+1}\|_{L^1_{\gamma,|v_1|,+}}&=\int_{v_1<0}f^{n+1}(0,v)|v_1|dv+\int_{v_1>0}f^{n+1}(1,v)|v_1|dv\cr
	&\leq \delta_1\|f_{LR}\|_{L^1_{\gamma,|v_1|}}+\delta_2\|M_w\|_{L^1_{\gamma,|v_1|,+}}+\delta_3\|f^n\|_{L^1_{\gamma,|v_1|,+}}\cr
	&+C_{\ell,u}\left(\frac{\ln \tau+1}{\tau}\right)\cr
		&\leq \delta_1\|f_{LR}\|_{L^1_{\gamma,|v_1|}}+\delta_2+\delta_3\|f^n\|_{L^1_{\gamma,|v_1|,+}}\cr
	&+C_{\ell,u}\left(\frac{\ln \tau+1}{\tau}\right).
	\end{split}
	\end{align}
	
	Therefore, in view of $\mathcal{D}_i $ of $\Omega_i$ $(i=1,2)$, we see that
	\begin{align}\label{outflux2}
	\begin{split}
	\|f^{n+1}\|_{L^1_{\gamma,|v_1|,+}}
	&\leq
	\delta_1\|f_{LR}\|_{L^1_{\gamma,|v_1|}}
	+\delta_2+2\delta_3\big(1+\|f_{LR}\|_{L^1_{\gamma,|v_1|}}\big)
	+C_{\ell,u}\left(\frac{\ln \tau+1}{\tau}\right)\cr
	&=2(\delta_1+\delta_2+\delta_3)\big(1+\|f_{LR}\|_{L^1_{\gamma,|v_1|}}\big)-(\delta_1+\delta_2)
	\big(1+\|f_{LR}\|_{L^1_{\gamma,|v_1|}}\big)\cr
	&+C_{\ell,u}\left(\frac{\ln \tau+1}{\tau}\right)\cr
	&\leq 2\big(1+\|f_{LR}\|_{L^1_{\gamma,|v_1|}}\big)
	\end{split}
	\end{align}
	for sufficiently large $\tau$.\newline

	\noindent $\bullet$ Estimate for influx: $\|f^{n+1}\|_{L^1_{\gamma,|v_1|,-}}$: When $v_1>0$, we have from the boundary condition (\ref{f_delta dff}) that
	\begin{align*}
	f^{n+1}(0,v)=\delta_1 f_L+\delta_2\mathcal{S}_LM_w(0)+\delta_3 f^n(0,Rv).
	\end{align*}
	Integrate both sides with respect to $|v_1|dv$ on $v_1>0$ to get
	\begin{align}\label{f1112}
	\begin{split}
	\int_{v_1>0}f^{n+1}(0,v)|v_1|dv= \delta_1\int_{v_1>0}f_L|v_1|dv+\delta_2+\delta_3\int_{v_1>0}f^{n}(0,Rv)|v_1|dv.
	\end{split}
	\end{align}
	where we used $\mathcal{S}_L<1$ and $\int_{v_1>0}M_w(0)|v_1|dv=1$.
	Similarly, we estimate 	
	\begin{align}\label{f1122}
	\begin{split}
	&\int_{v_1<0}f^{n+1}(1,v)|v_1|dv
	= \delta_1\int_{v_1<0}f_R|v_1|dv+\delta_2+\delta_3\int_{v_1>0}f^{n}(1,Rv)|v_1|dv.
	\end{split}
	\end{align}
	Combining (\ref{f1112}) and (\ref{f1122}) gives
	\begin{align}\label{ff11-2}
	\begin{split}
	\|f^{n+1}\|_{L^1_{\gamma,|v_1|,-}}
	&\leq \delta_1\|f_{LR}\|_{L^1_{\gamma,|v_1|}}+\delta_2+\delta_3\|f^n\|_{L^1_{\gamma,|v_1|,+}},
	\end{split}
	\end{align}
	which, thanks to (\ref{outflux2}), gives
	\begin{align*}
	\begin{split}
	\|f^{n+1}\|_{L^1_{\gamma,|v_1|,-}}
	&\leq 2\big(1+\|f_{LR}\|_{L^1_{\gamma,|v_1|}}\big).
	\end{split}
	\end{align*}
	This completes the proof of (1). The proof of (2) is identical. We omit the proof.
\end{proof}
\begin{lemma}\label{au dff} $(1)$ Let $f^n\in \Omega_3$ or $\Omega_4$. For sufficiently large $\tau>0$, we have
	\begin{align*}
		\int_{\mathbb{R}^3}f^{n+1}(1+|v|^2)dv\leq  2\left(\|f_{LR}\|_{L^1_{\gamma,|v_1|}}+\|M_w\|_{L^1_{\gamma,\langle v\rangle}}\right).
	\end{align*}
\end{lemma}
\begin{proof}
The proof is almost identical to Lemma \ref{au}. We omit it.	
\end{proof}
\begin{lemma}\label{23 dff}
	Let $f^n\in\Omega_3$ or $\Omega_4$. \newline
	\noindent$(1)$ For $i=1$, we have
	
	\begin{align*}
		\Big|\int_{\mathbb{R}^3}f^{n+1}v_1dv\Big|
		&\leq\delta_1\left|\int_{v_1>0}f_L|v_1|dv-\int_{v_1<0}f_R|v_1|dv\right|
		+\frac{2}{\tau}C_{LM,2}
	\end{align*}
	where $C_{LR,2}$ denotes
	\[
	C_{LM,2}=\|f_{LR}\|_{L^1_{\gamma,\langle v\rangle}}+\|M_w\|_{L^1_{\gamma,\langle v\rangle}}.
	\]
	\noindent$(2)$ For $i=2,3$, we have
	\begin{align*}
		\Big|\int_{\mathbb{R}^3}f^{n+1}v_idv\Big|\leq 2\delta_3C_{LR,2}
		+
		C\left(\frac{\ln \tau+1}{\tau}\right).
	\end{align*}
\end{lemma}
\begin{proof}
	\noindent$(1)$
Recall from the proof of Lemma \ref{23} that

	\begin{align*}
		\int_{\mathbb{R}^3}f^{n+1}(x,v) v_1dv&=I+II,
	\end{align*}
	where
	\begin{align*}
		I=\int_{v_1>0}f^{n+1}(0,v)|v_1|dv-\int_{v_1<0}f^{n+1}(1,v)|v_1|dv
	\end{align*}
	and
	\begin{align*}
		II&=\int^x_0\int_{v_1>0}\frac{\rho^n}{\tau}\left(\mathcal{M}_{\nu}(f^n)-f^{n+1}\right)dvdy
		+\frac{1}{\tau}\int^1_x\int_{v_1<0}\left(\mathcal{M}_{\nu}(f^n)-f^{n+1}\right)dydv.
	\end{align*}
	$(a)$ The estimate of $I$: 
	We observe from our boundary condition that
	\begin{align*}
		|I|&\leq \left(\delta_1+\frac{\delta_1\delta_2}{2-\delta_1}\right)\left|\int_{v_1>0}f_L|v_1|dv-\int_{v_1<0}f_R|v_1|dv\right|\cr
		&+\frac{1}{\tau(2-\delta_2)}\int_{\rrr}\int^1_0\left\{\mm(f^n)+f^n\right\}(1+|v|^2)dydv\cr
			&\leq2\delta_1\left|\int_{v_1>0}f_L|v_1|dv-\int_{v_1<0}f_R|v_1|dv\right|+\frac{2}{\tau}C_{LM,2}.
	\end{align*}
	In the last line, we used Lemma \ref{auu}.\newline
	
	\noindent$(b)$ The estimate for $II$: The argument for this part is identical
	except that we use Lemma \ref{au dff} instead of Lemma \ref{au}.
	Now, we combine $(a)$ and $(b)$ to obtain the desired result.\newline\newline
	\noindent (2) The proof is identical to the inflow dominant case, since the $\delta_2$ contribution vanishes:
	\begin{align*}
		\int_{\mathbb{R}^2}\mathcal{S}(f^{n})M_wv_2dv_2dv_3=0.
	\end{align*}
	We omit the proof.
\end{proof}
\begin{lemma}\label{gl dff} $(1)$ Let $-1/2<\nu<1$. Assume $f^n\in\Omega_3$.  Then, for sufficiently large $\tau$, we have
	\begin{align*}
		C^1_{\nu}\delta_2^2\frac{\gamma_{\ell,2}}{27C_{LR,2}^2}\leq\kappa^{\top}\left\{\mathcal{T}^{n+1}_{\nu}\right\}\kappa\leq \frac{2 }{{3a_{\ell,2}}}C^2_{\nu}C_{LM,2}
	\end{align*}
	
	\noindent$(2)$ Let $\nu=-1/2$ and $f\in \Omega_4$. Then, for sufficiently large $\tau$, we have
	\begin{align*}
		\delta_1\frac{a_{-1/2}}{4C_{LR,2}}\leq\kappa^{\top}\left\{\mathcal{T}^{n+1}_{-1/2}\right\}\kappa\leq \frac{3}{2a_{\ell,2}}C_{LR,2}
	\end{align*}
	for any $\kappa\in\mathbb{R}$ and $|\kappa|=1$. We recall that 
	\[
	C_{LM,2}=\|f_{LR}\|_{L^1_{\gamma,\langle v\rangle}}+\|M_w\|_{L^1_{\gamma,\langle v\rangle}}.
	\]
\end{lemma}

\begin{proof}
	(1) The proof is identical to the inflow dominant case, except for the computation of $I$,
	 where we bound it from below using $\delta_2$ and $\gamma_{\ell,2}$, instead of using 
	 $\delta_1$ and $\gamma_{\ell,1}$.
	\begin{align*}
	I
	&=4\left\{\int_{v_1>0}f^{n+1}|v_1|dv\right\}\left\{\int_{v_1<0}f^{n+1}|v_1|dv\right\}\cr
	&\geq\delta_2^2\left(\int_{v_1>0}e^{-\frac{1}{|v_1|}}M_w(0)|v_1|dv\right)
	\left(\int_{v_1<0}e^{-\frac{1}{|v_1|}}M_w(1)|v_1|dv\right)\cr
	&=\delta^2_2\gamma_{\ell,2}.
	\end{align*}
	In the last line, we used  
	\begin{align*}
	f^{n+1}&\geq \frac{1}{3}\delta_2e^{-\frac{1}{\tau|v_1|}}M_w(0)1_{v_1>0}+\frac{1}{3}\delta_2
	e^{-\frac{1}{\tau|v_1|}}M_w(1)1_{v_1<0}\cr
	\end{align*}
	and $\tau>1$.
\newline
	%
	%
	\noindent(2) We recall from the inflow dominant case that
	\begin{align}\label{tbt dff}
	\begin{split}
	&\rho^{n+1}\kappa^{\top}\left\{\mathcal{T}^{n+1}_{-1/2}\right\}\kappa\cr
	&\quad=\int_{\mathbb{R}^3}f^{n+1}\left\{|v|^2-(v\cdot\kappa)^2\right\}dv-\left\{\rho^{n+1}
	|U^{n+1}|^2-\rho^{n+1} (U^{n+1}\cdot \kappa)^2\right\}\cr
	&\quad\equiv I+II,
	\end{split}
	\end{align}
	for $|\kappa|=1$\newline
	
	\noindent$(i)$ Upper bound: We have from Lemma \ref{au dff} that
	\begin{align*}
	\kappa^{\top}\left\{\mathcal{T}^{n+1}_{-1/2}\right\}\kappa\leq \frac{1}{\rho^{n+1}}\int_{\mathbb{R}^3}f^{n+1}|v|^2dv\leq\frac{1}{a_{\ell,2}\delta_2}
	\left(1+\|f_{LR}\|_{L^1_{\gamma,\langle v\rangle}}\right)\|M_w\|_{L^1_{\gamma,\langle v\rangle}}.
	\end{align*}
	\noindent$(ii)$ Lower bound: For this, we estimate the lower bound of $I$ and the smallness of $II$:\newline
	\noindent$(ii$-$a)$ Lower bound of $I$: The proof is the same, except that we bound it using $a_{-1/2,2}$ this time:
	\begin{align*}
	I
	\geq\frac{1}{3}\delta_2\inf_{|\kappa|=1}\int_{\mathbb{R}^3}e^{-\frac{1}{|v_1|}}
		M_w\left\{|v|^2-(v\cdot\kappa)^2\right\}dv
		=\frac{1}{3}\delta_2 a_{-1/2,2}.
	\end{align*}

	\noindent$(ii$-$b)$ Smallness of of $II$: The estimate for this case is the same either, except that we use Lemma  \ref{23 dff}, instead of Lemma \ref{23}: 
	\begin{align*}
	II&\leq \frac{|\rho^{n+1} U^{n+1}|^2}{\rho^{n+1}}
	\leq\frac{1}{a_{\ell,1}}\left|\int_{\mathbb{R}^3}f^{n+1}vdv\right|^2
	\leq \frac{1}{a_{\ell,1}}\sum_{i=1}^3\left|\int_{\mathbb{R}^3}f^{n+1}v_idv\right|^2\cr
	&\leq4\delta_1^2\left|\int_{v_1>0}f_L|v_1|dv-\int_{v_1<0}f_R|v_1|dv\right|^2+
	O\left(\delta_3,\tau^{-1}\right).
	\end{align*}
The by exactly the same argument, we get the desired result. Note that, since we can take $\delta_1$ arbitrariliy small in this case, we don't need to assume that the discrepancy of the flux from the inflow data is small.
\end{proof}
\subsection{Cauchy estimate for $f^n$}

\begin{proposition}\label{Lipshitz dff} Let $f$, $g$ be elements of $\Omega_3$ $(-1/2<\nu<1)$ or  $\Omega_4$ $(\nu=-1/2)$. Then the non-isotropic Gaussian $\mathcal{M}_{\nu}$ satisfies 
	\begin{eqnarray*}
		|\mathcal{M}_{\nu}(f)-\mathcal{M}_{\nu}(g)|\leq C\sup_x\|f-g\|_{L^{1}_2}
		e^{-C|v|^2}.
	\end{eqnarray*}
\end{proposition}
\begin{proof}
The proof is almost identical with the one given for Proposition \ref{Lipshitz}. We omit it.
\end{proof}

\begin{proposition}\label{fixed point 2 dff}
	Suppose $f^{n},\, f^{n+1}\in\Omega_{i}$ $(i=3,4)$. Then, under the assumption of Theorem \ref{Main2}, we have 
	\begin{align*}
		&\sup_{x}\|f^{n+1}-f^{n}\|_{L^1_{2}}+\|f^{n+1}-f^n\|_{L^1_{\gamma, |v_1|}}
		+\|f^{n+1}-f^n\|_{L^1_{\gamma, \langle v\rangle}}\cr
		 &\hspace{0.6cm}\preceq \left(\frac{\ln\tau+1+\delta_2}{\tau}\right)\sup_{x}\|f^{n}-f^{n-1}\|_{L^1_2}
		 +\delta_3\|f^{n}-f^{n-1}\|_{L^1_{\gamma, |v_1|}}+\delta_3\|f^{n}-f^{n-1}\|_{L^1_{\gamma, \langle v\rangle}}.
	\end{align*}
\end{proposition}
\begin{remark}
We note that, unlike in Proposition \ref{fixed point 2}, $K$ does not have $\|f_{LR}|v|^{-1}\|_{L^1_{\gamma,\langle v\rangle}}$ term in this case. This is why we don't need the no-concentration assumption ($P_2$) in Theorem \ref{Main2}.
\end{remark}
\begin{proof}
We only consider the boundary terms in  $\|\cdot\|_{L^1_{\gamma,|v_1|}}$ estimate. We note from our boundary condition that, for $v_1>0$
	\begin{align}\label{sum1 dff}
		\begin{split}
			&\int_{v_1>0}|f^{n+1}(0,v)-f^n(0,v)||v_1|dv\cr
			&\hspace{1.2cm}\leq\delta_2|\mathcal{S}^+(f^n)-\mathcal{S}(f^{n-1})|\int_{v_1>0}M_w(0)|v_1|dv\cr
			&\hspace{1.2cm}+\delta_3\int_{v_1>0}|f^n(0,Rv)-f^{n-1}(0,Rv)||v_1|dv\cr
			&\hspace{1.2cm}\leq C\frac{\delta_2}{\tau}\sup_x\|f^n-f^{n-1}\|_{L^1_2}
			+\delta_3\int_{v_1<0}|f^n(0,v)-f^{n-1}(0,v)||v_1|dv,
		\end{split}
	\end{align}
where we used Proposition \ref{Lipshitz dff} as
\begin{align*}
	\begin{split}
		|\mathcal{S}^+(f^{n})-\mathcal{S}(f^n)|&=
		\frac{1}{\tau(2-\delta_1)}\int_{v_1>0}\int^1_0\left|\mathcal{R}^n(y,v)-\mathcal{R}^{n-1}(y,v)\right|dydv\cr
		&\leq \frac{C}{\tau}\sup_x\|f^n-f^{n-1}\|_{L^1_2}.
	\end{split}
\end{align*}

	On the other hand, for $v_1<0$, we have from (\ref{f- dff})
	\[
	f^{n+1}(0,v)=I(f^n)+II(f^n),
	\]
	where 
	\begin{eqnarray*}
		I(f)=e^{-\frac{1}{\tau|v_1|}\int^x_0\rho_{f}(y)dy}f(1,v),\quad
		II(f)=\frac{1}{\tau|v_1|}\int^x_0e^{-\frac{x-y}{\tau|v_1|}}\mathcal{M}_{\nu}(h)dy.
	\end{eqnarray*}

Since
	\begin{align*}
		I(f^{n})-I(f^{n-1})
		&+\delta_2 e^{-\frac{x}{\tau|v_1|}}
		\left\{\mathcal{S}^-(f^n)-\mathcal{S}^-(f^{n-1})\right\}M_w(1)\cr
		&+\delta_3 e^{-\frac{x}{\tau|v_1|}}\left\{f^n(1,Rv)-f^{n-1}(1,Rv)\right\},
	\end{align*}

	we have 
	\begin{align}\label{sum2 dff}
		\begin{split}
			&\int_{v_1<0}|I(f^{n})-I(f^{n-1})|v_1|dv\cr
			&\qquad\leq C\frac{\delta_2}{\tau}\sup_x\|f^n-f^{n-1}\|_{L^1_2}+\delta_3\left\{\int_{v_1>0}|f^n(1,v)-f^{n-1}(1,v)||v_1|dv\right\}.
		\end{split}
	\end{align}
Now, through an almost identical computations as in the inflow dominant case, we get the following estimates:
	\begin{align}\label{final1 dff}
		\begin{split}
			\|f^{n+1}-f^{n}\|_{L^1_{\gamma,|v_1|}}
			&\preceq \left(\frac{\ln\tau+1+\delta_2}{\tau}\right)\sup_{x}\|f^{n}-f{n-1}\|_{L^1_{2}}
			+\delta_3\|f^{n}-f^{n-1}\|_{L^1_{\gamma,|v_1|}}.
		\end{split}
	\end{align}
Estimates in $\|\cdot\|_{L^1_{\gamma,\langle v\rangle}}$ and  $\|\cdot\|_{L^1_{2}}$ can be obtained similarly:
	\begin{align}\label{final2 dff}
		\begin{split}
			\|f^{n+1}-f^{n}\|_{L^1_{\gamma,\langle v\rangle}}
			&\preceq\left(\frac{\ln\tau+1+\delta_2}{\tau}\right)\sup_{x}\|f^{n}-f^{n-1}\|_{L^1_{2}}			
			+\delta_3C\|f^{n}-f^{n-1}\|_{L^1_{\gamma,\langle v\rangle}}
		\end{split}
	\end{align}
and
	\begin{align}\label{final3 dff}
		\begin{split}
			\|f^{n+1}-f^{n}\|_{L^1_2}
			&\preceq \left(\frac{\ln\tau+1+\delta_2}{\tau}\right)\sup_{x}\|f^{n}-f^{n-1}\|_{L^1_{2}}
		+\delta_3\|f^{n}-f^{n-1}\|_{L^1_{\gamma,\langle v\rangle}}.
		\end{split}
	\end{align}
	The estimates (\ref{final1 dff}), (\ref{final2 dff}) and (\ref{final3 dff}) give the desired result.
\end{proof}

\noindent{\bf Acknowledgement}\newline
Part of the this work is done while S.-B. Yun was visiting 
the institute of Mathematics at University of Bordeaux.
S.-B. Yun would like to acknowledge the hospitality of the institute.
St\'ephane Brull is supported by the French-Korean IRL FK maths.
Seok-Bae Yun is supported by Samsung Science and Technology Foundation under Project Number SSTF-BA1801-02.

%
%
%
%

\bibliographystyle{amsplain}

\end{document}